\documentclass[12pt]{article}
\usepackage{amsmath,amsthm,amssymb,latexsym,color, dsfont}
\theoremstyle{plain}
\newtheorem{theor}{Theorem}[section]
\newtheorem*{theoA}{Theorem~A}
\newtheorem*{theoA'}{Theorem~$A^\star$}
\newtheorem*{theoB}{Theorem~B}
\newtheorem{prop}[theor]{Proposition}
\newtheorem{definition}[theor]{Definition}

\newtheorem*{corA}{Corollary~A}
\newtheorem{lemma}[theor]{Lemma}
\theoremstyle{remark}
\newtheorem{rem}[theor]{Remark}

\DeclareMathOperator{\Comp}{Comp}
\DeclareMathOperator{\Incomp}{Incomp}

\DeclareMathOperator{\LCD}{LCD}

\def\R{{\mathbb R}}
\def\P{{\mathbb P}}
\def\E{{\mathbb E}}
\def\F{{\mathcal F}}

\def\N{{\mathbb N}}
\def\Z{{\mathbb Z}}
\def\K{{\mathcal K}}
\def\Net{{\mathcal N}}

\def\dist{{\rm dist}}

\def\Exp{{\mathbb E}}

\def\Vol{{\rm Vol}}
\def\supp{{\rm supp}}

\def\Diag{\mathcal D}
\def\DiagDyadic{{\mathcal D}^{2}}
\def\eps{\varepsilon}

\def\Event{{\mathcal E}}
\def\Levy{{\mathcal L}}
\def\CNumber{{\bf N}}

\begin{document}

\title{
Coverings of random ellipsoids, and invertibility of matrices with i.i.d.\ heavy-tailed entries
}
\author{Elizaveta Rebrova\thanks{E.R. was partially supported by U.S. Air Force grant F035062.}\quad
and\quad Konstantin Tikhomirov\thanks{K.T. was partially supported by PIMS
Graduate Scholarship and by Dean's Excellence Award,
Faculty of Science, UofA.}
}

\newcommand\address{\noindent\leavevmode

\medskip

\noindent
Elizaveta Rebrova\\
Department of Mathematics,\\
University of Michigan,\\
Ann Arbor, MI  48109-1043\\
\texttt{\small e-mail: erebrova@umich.edu}

\medskip

\noindent
Konstantin Tikhomirov\\
Dept.~of Math.~and Stat.~Sciences,\\
University of Alberta, \\
Edmonton, Alberta, Canada, T6G 2G1.\\
Current address: Dept.\ of Mathematics, Princeton University, Princeton NJ 08544\\
\texttt{\small e-mail: kt12@math.princeton.edu}
}

\maketitle

\begin{abstract}
Let $A=(a_{ij})$ be an $n\times n$
random matrix with i.i.d.\ entries such that $\E a_{11} = 0$ and $\E {a_{11}}^2 = 1$.
We prove that for any $\delta>0$ there is $L>0$ depending only on $\delta$,
and a subset $\Net$ of $B_2^n$ of cardinality at most $\exp(\delta n)$
such that with probability very close to one we have
$$A(B_2^n)\subset \bigcup_{y\in A(\Net)}\bigl(y+L\sqrt{n}B_2^n\bigr).$$
In fact, a stronger statement holds true. As an application, we show
that for some $L'>0$ and $u\in[0,1)$
depending only on the distribution law of $a_{11}$, the smallest singular value $s_n$
of the matrix $A$ satisfies
$\P\{s_n(A)\le \varepsilon n^{-1/2}\}\le L'\varepsilon+u^n$ for all $\varepsilon>0$.
The latter result generalizes a theorem of Rudelson and Vershynin which
was proved for random matrices with subgaussian entries.
\end{abstract}

\section{Introduction}

In this paper, we consider random matrices $A$ satisfying
\begin{equation}\tag{*}\label{probab model}
\mbox{$A$ is $n\times n$; the entries of $A$ are i.i.d., with $\Exp a_{ij}=0,\;\;\Exp{a_{ij}}^2=1$.}
\end{equation}
We are concerned with the following question:
{\it how many translates of a Euclidean ball $\sqrt{n}B_2^n$ (or its constant multiple)
are needed to cover the random ellipsoid
$A(B_2^n)$?} Being geometrically natural,
this problem, as we will see later, has an application to studying
invertibility properties of the matrix $A$.

If the entries of $A$ have a bounded fourth moment then the operator norm $\|A\|_{2\to 2}$
satisfies $\|A\|_{2\to 2}\leq L\sqrt{n}$
with probability close to one (see
\cite{YBK} and \cite{Latala} for precise statements),
whence $\P\{A(B_2^n)\subset L\sqrt{n}B_2^n\}\approx 1$.
If, moreover, the entries of $A$ are subgaussian then
for some $L>0$ depending only on the subgaussian moment we have
$\P\{A(B_2^n)\subset L\sqrt{n}B_2^n\}\geq 1-\exp(-n)$.
On the other hand, for heavy-tailed entries the operator norm of $A$
may have a higher order of magnitude compared to $\sqrt{n}$ with probability close to one,
so the trivial argument given above is not applicable.
The first main result of the paper is the following theorem:
\begin{theoA}
Let $\delta\in(0,1/4]$ and $n\geq \frac{1}{4\delta}$.
Then there is a (non-random) collection $\mathcal C$ of parallelepipeds in $\R^n$ with
$|\mathcal C|\leq \exp(13 n \delta \ln\frac{2e}{\delta})$ having the following property:
For any random matrix $A$ satisfying \eqref{probab model},
with probability at least $1-4\exp(-\delta n/8)$ we have
$$\forall \;x\in B_2^n\;\;\;\;\exists P\in\mathcal C\mbox{ such that } x\in P\mbox{ and }A(P)\subset Ax+\frac{C\sqrt{n}}{\delta}B_2^n.$$
Here, $C>0$ is a universal constant.
\end{theoA}

In particular, the above theorem implies the following more elegant 
\begin{corA}
For any $\delta\in(0,1/4]$ and $n\geq \frac{1}{4\delta}$
there exists a non-random subset $\Net\subset B_2^n$
of cardinality at most $\exp(13 n \delta \ln\frac{2e}{\delta})$ such that for any $n\times n$ matrix $A$
satisfying \eqref{probab model}, we have
$$\P\Big\{A(B_2^n)\subset \bigcup_{y\in A(\Net)}\bigl(y+\frac{C'\sqrt{n}}{\delta}B_2^n\bigr)\Big\}\geq 1-4\exp(-\delta n/8)$$
for some universal constant $C'>0$.
\end{corA}
Both results have geometric interpretation in terms of covering numbers. Recall that for two subsets $S$ and $K$ of a vector space the \emph{covering number}
$\CNumber(S,K)$ is defined as the smallest number of parallel translates of $K$ sufficient to cover $S$. By Theorem~A, $\CNumber(A(B_2^n), \frac{C\sqrt{n}}{\delta} B_2^n)\leq \exp\bigl(13\delta n\ln\frac{2e}{\delta}\bigr)$
with probability at least $1-4\exp(-\delta n/8)$. 

Another interpretation of these results, that will be of use for us, is related to the net refinement (see Theorem~A* in Section~5).
Given a metric space $X$, an $\varepsilon$\emph{-net} $\Net$ on $X$ is a subset of $X$ such that any point of $X$ is within distance at
most $\varepsilon$ from a point of $\Net$.
It is easy to see that with probability at least $1-4\exp(-\delta n/8)$ the set $\Net$ from Corollary~A is a $\frac{C\sqrt{n}}{\delta}$--net
on $B_2^n$ with respect to the pseudometric $d(x,y):=\|A(x-y)\|$ ($x,y\in B_2^n)$. Here and further, $\|\cdot\|$ denotes the
standard Euclidean norm in $\R^n$.

A crucial feature of these results is that the set $\mathcal C$ in the theorem is non-random.
Moreover, $\mathcal C$ (as well as the set $\Net$ from Corollary~A)
provides a ``universal'' covering which is independent of the distribution of the entries of $A$. 

Finally, compared to Corollary~A, the statement of Theorem~A is more flexible as it enables us
to choose the ``anchor'' points within the parallelepipeds when constructing corresponding $\varepsilon$-net
(this matter is covered in detail at the beginning of Section~\ref{ssv section}).

\bigskip 

Let us briefly describe the main idea of the proof.
The collection $\mathcal C$ of parallelepipeds is constructed using a special subset
${\mathcal D}$ of diagonal operators with diagonal elements
in the interval $(0,1]$. Namely, we define ${\mathcal D}$
as the set of all diagonal operators with diagonal entries in $\{1\}\cup\{2^{-2^k}\}_{k=0}^\infty$
and with determinants bounded from below by $\exp(-\delta n)$.
Then, for every operator $D$ from ${\mathcal D}$, we take a covering of the ball $B_2^n$
by appropriate translates of parallelepiped $D(L'' n^{-1/2} B_\infty^n)$
(for some $L''=L''(\delta)$),
and let $\mathcal C$ be the union of such coverings over ${\mathcal D}$.
It turns out that Theorem~A follows almost immediately from the following relation:
\begin{equation}\label{intro bound of infty to 2}
\begin{split}
\P\Bigl\{&\exists\mbox{ diagonal matrix }D\mbox{ with diagonal entries in }\{1\}\cup\{2^{-2^k}\}_{k=0}^\infty
\mbox{ such that}\\
&\det D\ge\exp(-\delta n)\mbox{ and }\|AD\|_{\infty\to 2}\le \frac{Cn}{\sqrt{\delta}}\Bigr\}\ge 1-4\exp(-\delta n/8).
\end{split}
\end{equation}
In Section~\ref{parallel section}, we show that \eqref{intro bound of infty to 2}
holds true under condition \eqref{probab model}; see Theorem~\ref{parallelepiped norm estimate}.
Geometrically, this property means that
it is possible to construct a random parallelepiped $P\subset[-1,1]^n$ with sides parallel to
the standard coordinate axes, such that $\Vol (P)\ge \exp(-\delta n)$
and $A$ maps $P$ inside the Euclidean ball $\frac{Cn}{\sqrt{\delta}}B_2^n$ with probability at least $1-4\exp(-\delta n/8)$.
Note that parallelepiped $P$ will be ``narrow'' along directions $w\in S^{n-1}$ for which $\|Aw\|$ is large.

\bigskip

As we already mentioned above, Theorem~A has a direct application to the problem
of obtaining quantitative (non-asymptotic) estimates for the smallest singular value of $A$.
Recall that, given an $m\times n$ ($m\ge n$) matrix $M$, its smallest singular value can be defined as
$s_{n}(M)=\inf\limits_{y\in S^{n-1}}\|My\|$. An argument based on Theorem~A and results of Rudelson
and Vershynin from \cite{RV square, RV rectangular}, yields:
\begin{theoB}
For any $\widetilde v\in(0,1]$ and $\widetilde u\in(0,1)$ there are
numbers $L>0$, $u\in(0,1)$ and $n_0\in\N$ depending only on $\widetilde v$ and $\widetilde u$ with the following property:
Let $n\ge n_0$ and let $A=(a_{ij})$ be an $n\times n$ random matrix satisfying \eqref{probab model},
such that $\sup\limits_{\lambda\in\R}\P\{|a_{11}-\lambda|\le \widetilde v\}\le \widetilde u$.
Then for any $\varepsilon>0$ we have 
$$
\P\bigl\{s_n(A)\le\varepsilon n^{-1/2}\bigr\} \le L \varepsilon + u^n.
$$
\end{theoB}
Note that any random variable $\alpha$ with $\Exp\alpha=0$ and $\Exp \alpha^2=1$
obviously satisfies 
$\sup\limits_{\lambda\in\R}\P\{|\alpha-\lambda|\le \widetilde v\}\le \widetilde u$ for some
$\widetilde v>0$ and $\widetilde u\in(0,1)$ determined by the law of $\alpha$. Thus, the above statement does not require any additional assumptions on the matrix apart from
\eqref{probab model};
by introducing the quantities $\widetilde v$ and $\widetilde u$ we make the dependence of $L$ and $u$ on
the law of $a_{11}$ more explicit.

Let us put Theorem~B in the context of known results.

Convergence of (appropriately normalized) smallest singular values for a sequence of random rectangular matrices
with i.i.d.\ entries and growing dimensions was established by Bai and Yin \cite{BY}
(see also \cite{T limit}, where the result is proved under optimal moment assumptions).
For non-asymptotic results in this direction,
we refer the reader to papers \cite{LPRT, RV rectangular} for the case of i.i.d.\ entries
(see also \cite{Tikhomirov} where no moment conditions are assumed);
\cite{ALPT, ALPT2} for log-concave distributions of rows
and \cite{SV, MP, KM, Yas14, GLPT} for more general isotropic distributions.
We refer to surveys \cite{RV congress, Vershynin} (see also \cite{Rudelson})
for more information.

For random square matrices with independent standard Gaussian entries,
the limiting distribution of the smallest singular value was computed by Edelman \cite{Edelman};
universality of this result was established in \cite{TV distribution of smin}.
Further, for matrices with i.i.d.\ entries
it was shown in \cite{Tao Vu Circular law}
and \cite{TV smooth analysis} that, given any $K>0$ there are $R,L>0$ depending only on $K$ and the law of $a_{11}$
such that $\P\{s_n(A+B)\le n^{-L}\}\le Rn^{-K}$ for any non-random matrix $B$ satisfying $\|B\|_{2\to 2}\le n^K$
(we note that analogous results were recently obtained for more general
models of randomness allowing some dependence between the entries of $A$;
see, in particular, \cite{NR} and \cite{GNT}).
In the case $B={\bf 0}$ which we study in this paper, those papers do not provide
optimal estimates for $s_n(A)$. A much more precise statement was proved in \cite{RV square}
under the additional assumption that the entries of $A$ are subgaussian; namely,
Rudelson and Vershynin showed that
$s_n(A)$ satisfies a small ball probability estimate
$$\P\{s_n(A)\le \varepsilon n^{-1/2}\}\le L\varepsilon+u^n,\;\;\varepsilon>0,$$
where $L>0$ and $u\in(0,1)$ depend only on the subgaussian moment of $a_{ij}$'s. Note that Theorem~B gives an estimate of exactly the same form, but for the matrices with heavy-tailed entries.

\bigskip

The idea of the proof of Theorem~B can be described as follows. Denote by $A'$ the transpose of the first $n-1$ columns of $A$. A principal component of the proof of \cite{RV square} is an analysis of the arithmetic structure
of null vectors of $A'$,
which is described with the help of the notion of the least common denominator (LCD).
To show that null vectors of $A'$ typically have an exponentially large LCD,
the authors of \cite{RV square}
consider subsets $S$ of the unit sphere corresponding to vectors with small LCD,
and show that $\inf\limits_{x\in S}\|A'x\|>0$ with a large probability.
For this, they use the standard $\varepsilon$-net argument,
when the infimum is estimated by
taking a Euclidean $\varepsilon$-net $\Net$ on $S$ and applying relation
$\inf\limits_{x\in S}\|A'x\|\ge\inf\limits_{y\in \Net}\|A'y\|-\varepsilon\|A'\|_{2\to 2}$
together with the estimate $\|A'\|_{2\to 2}\leq C\sqrt{n}$ which holds with probability
very close to one under the subgaussian
moment assumptions on the entries.
In our setting, the principal difficulty consists in the fact that the condition \eqref{probab model} does
not guarantee a good upper bound for the operator norm $\|A'\|_{2\to 2}$. To deal with this fundamental issue,
we ``refine'' the nets constructed in \cite{RV square} by applying Theorem~A.
Indeed, it can be shown that Theorem~A implies that, given an $\varepsilon$-net $\Net$ on
$S$, it is possible to construct a subset $\widetilde \Net\subset S$ of cardinality
at most $\exp\bigl(13\delta n\ln\frac{2e}{\delta}\bigr)|\Net|$ which is an $L'\varepsilon \sqrt{n}$-net
on $S$ (for some $L'=L'(\delta)$) with respect to the pseudometric $d(x,y) = \|A'(x-y)\|$
with probability at least $1-4\exp(-\delta n/8)$. Then,
$\inf\limits_{x\in S}\|A'x\|\ge\inf\limits_{y\in \widetilde\Net}\|A'y\|-L'\varepsilon\sqrt{n}$,
so the argument does not depend any more on the value of $\|A'\|_{2\to 2}$.

\bigskip

The paper is organized as follows:
Sections~\ref{vector section} and~\ref{parallel section} are devoted to proving the main novel element of the paper
--- Theorem~A. Then, in Section~\ref{sec prelim}, we collect
some results from \cite{RV square}, and, in Section~\ref{ssv section}, prove Theorem~B.

Finally, let us discuss notation.
Given a finite set $S$, by $|S|$ we denote its cardinality.
By $e_1,e_2,\dots,e_n$ we denote the canonical basis in $\R^n$.
The standard inner product in $\R^n$
shall be denoted by $\langle\cdot,\cdot\rangle$.
Given $p\in[1,\infty]$, $\|\cdot\|_p$ is the standard $\ell_p$-norm.
For $\ell_2$, will will simply write $\|\cdot\|$.
Given an $m\times n$ matrix $M$ and $p,q\in[1,\infty]$, by $\|M\|_{p\to q}$ we shall denote the operator norm of $M$
considered as the mapping from $(\R^n,\|\cdot\|_p)$ to $(\R^m,\|\cdot\|_q)$.
Universal positive constants shall be denoted by $C,c$. Sometimes, to avoid confusion, we shall add a
numerical subscript to the name of a constant or function defined within a statement.

\section{Fitting a random vector into an $\ell_p^n$-ball}\label{vector section}

Throughout the paper, by $\Diag_n$ we denote the set of all $n\times n$ diagonal matrices with
diagonal elements belonging to the interval $(0,1]$ (we will sometimes refer to such matrices
as positive diagonal contractions). Further,
denote by $\DiagDyadic_n$ the set of all $n\times n$ positive diagonal contractions whose diagonal entries belong to the set
$\{1\}\cup\{2^{-2^{k}}\}_{k=0}^\infty$.
The set $\DiagDyadic_n$ can be regarded as a discretization of $\Diag_n$.
 
 \bigskip
 
In this section, we consider the following problem:
Let $X$ be a random vector in $\R^n$ with i.i.d.\ coordinates.
We want to find a random diagonal operator $D$ taking values in $\Diag_n$ such that
$D(X)$ is contained in an appropriate (fixed) multiple of the $\ell_p^n$-ball {\it everywhere} on the probability space
and at the same time the determinant of $D$ is typically ``not too small''.
The statement to be proved is
\begin{prop}\label{vector theorem}
For any $\alpha\in(0,1)$ there is a number $L=L(\alpha)>0$ with the
following property:
Let $\delta\in(0,1]$, $p\in[1,\infty)$ and let $X=(x_1,x_2,\dots,x_n)$ be a random vector on $(\Omega,\Sigma,\P)$ with
i.i.d.\ coordinates such that $\Exp |x_i|^p<\infty$.
Then there is a random positive diagonal contraction $D$ taking values in $\Diag_n$ such that
$$\|DX\|_p^p\le \frac{L}{\delta}\Exp\|X\|_p^p\;\;\mbox{everywhere
on the probability space, and}\;\;\Exp(\det D)^{p\alpha-p}\le\exp(\delta).$$
\end{prop}
\begin{rem}
Proposition~\ref{vector theorem} is a foundation block of our paper. In Section~\ref{parallel section},
we will amplify this result (the case $p=2$) by proving its ``matrix version'' (Theorem~\ref{parallelepiped norm estimate}).
The case $p\neq 2$ in this section is considered just for completeness.
\end{rem}
\begin{rem}
Note that a trivial definition of the diagonal operator $D=(d_{ij})$ by setting
$${d_{jj}}^p:=\min\Bigl(1,\frac{L}{\delta}\frac{\Exp\|X\|_p^p}{\|X\|_p^p}\Bigr),\;\;j=1,2,\dots,n,$$
gives an unsatisfactory distribution of the determinant.
For example, if the entries of $X$ are $\{0,1\}$-valued with probability of taking value $1$ equal to $1/n$,
then $\Exp\|X\|_p^p=1$, and for any $m\le n$ we have
$$\P\{\|X\|_p^p= m\}={n\choose m}n^{-m}\Bigl(1-\frac{1}{n}\Bigr)^{n-m}
\ge \frac{1}{4m^m}.$$
Thus, the above definition of $D$ would give $\P\{\det D\le 2^{-n}\}
\ge \frac{1}{4}\lceil 2^pL/\delta\rceil^{-\lceil 2^pL/\delta\rceil}$.

\medskip

Our construction of the required operator is more elaborate.
Let us first describe the idea informally. 
Assume that $p=1$ and that $X$ is our random vector with non-negative i.i.d.\ coordinates with unit expectations.
We consider a sequence of non-negative numbers ({\it levels}) such that
each coordinate exceeds $k$-th level with probability $2^{-k}$.
The main observation is that $X$ ``does not fit'' into the $\ell_1^n$-ball $\frac{Ln}{\delta}B_1^n$ only if for some $k$ there are
much more than $2^{-k}n$ coordinates of $X$ exceeding the level. We define
the required operator $D$ so that its restriction to the ``bad'' coordinates is an appropriate
dilation, while on all other coordinates it acts isometrically.
If there exist several ``bad'' levels the operator $D$ will be defined as a product of several diagonal operators. Moreover,
it will be more convenient to ``replace''
the vector $X$ by a sum of independent vectors of two-valued variables, such that the sum is
a majorant for $X$ on the entire probability space. We construct the majorant in the
coupling Lemma~\ref{majorant lemma} stated below.
\end{rem}

Given a non-negative random variable $\xi$ with an everywhere continuous
cumulative distribution function (in particular, $\P\{\xi=0\}=0$), define numbers $\tau_k(\xi)$ {\it (levels)} as
$$\tau_k(\xi):=\inf\bigl\{\tau\ge 0:\P\{\xi\ge\tau\}= 2^{-k}\bigr\},\;\;k\ge 0.$$
Note that
\begin{equation}\label{expectation in terms of tau}
\Exp\xi\ge\sum\limits_{k=0}^\infty 2^{-k-1}\tau_k(\xi).
\end{equation}

We will need the following standard fact:
\begin{lemma}[{see, for example, \cite[Chapter~1, Theorem~3.1]{Thorisson}}]\label{coupling fact}
Let $\xi_1,\xi_2$ be two random variables on a probability space $(\Omega,\Sigma,\P)$, and assume that
$\P\{\xi_1> t\}\geq \P\{\xi_2>t\}$ for all $t\in\R$ (that is, $\xi_2$ is stochastically dominated by $\xi_1$).
Then there is a probability space $(\widetilde\Omega,\widetilde\Sigma,\widetilde\P)$ and random variables
$\widetilde\xi_1,\widetilde\xi_2$ on $(\widetilde\Omega,\widetilde\Sigma,\widetilde\P)$
such that 1) $\widetilde \xi_i$ is equidistributed with $\xi_i$, $i=1,2$, and 2) $\widetilde \xi_1\geq \widetilde \xi_2$
everywhere on $\widetilde\Omega$.
\end{lemma}

\begin{lemma}[Coupling]\label{majorant lemma}
Let $Y=(y_1,y_2,\dots,y_n)$ be a random vector on a probability space $(\Omega,\Sigma,\P)$
with i.i.d.\ non-negative coordinates with everywhere continuous cdf and
$\Exp y_i=1$. Further, let $\xi_{i}^k$ ($i\leq n$, $k=0,1,\dots$) be $0\text{-}1$ variables on $(\Omega,\Sigma,\P)$ with
$\P\{\xi_{i}^k=1\}=2^{-k}$, and such that $\xi_i^k$ are jointly independent for all $i\leq n$ and $k\geq 0$,
and set
$$z_i:=\sum\limits_{k=0}^\infty \tau_{k+1}(y_i)\xi_{i}^k,\;\;i=1,2,\dots,n,$$
and $Z:=(z_1,z_2,\dots,z_n)$.
Then there is a probability space
$(\widetilde\Omega,\widetilde\Sigma,\widetilde\P)$ and random vectors
$\widetilde Y=(\widetilde y_1,\widetilde y_2,\dots,\widetilde y_n)$ 
and $\widetilde Z=(\widetilde z_1,\widetilde z_2,\dots,\widetilde z_n)$ on $(\widetilde\Omega,\widetilde\Sigma,\widetilde\P)$
such that
\begin{enumerate}
\item[a)] $\widetilde Y$ and $\widetilde Z$ are equidistributed with $Y$ and $Z$, respectively;
\item[b)] $\widetilde z_{i}\ge \widetilde y_{i}$ for all $i\leq n$ everywhere on $(\widetilde\Omega,\widetilde\Sigma,\widetilde\P)$.
\end{enumerate}
\end{lemma}
\begin{proof}
Fix for a moment $i\leq n$ and consider the distributions of $y_i$ and $z_i$.
Take any $t>0$. If $\tau_k(y_i)\le t$ for all $k\ge 0$ then, obviously,
$$\P\{z_i\ge t\}\ge 0=\P\{y_i\ge t\}.$$
Otherwise, let $k(t):=\max\{k\ge 0:\, \tau_k(y_i)\le t\}$. Then
$$\P\{z_i\ge t\}
\ge\P\{\tau_{k(t)+1}(y_i)\xi_i^{k(t)}\ge\tau_{k(t)+1}(y_i)\}
=2^{-k(t)}=\P\{y_i\ge\tau_{k(t)}(y_i)\}\ge\P\{y_i\ge t\}.$$
Thus, $y_i$ is stochastically dominated by $z_i$ and, by Lemma~\ref{coupling fact}, there is a 
probability space $(\widetilde\Omega_i,\widetilde\Sigma_i,\widetilde\P_i)$
and variables $\widetilde y_i$ and $\widetilde z_i$ on $(\widetilde\Omega_i,\widetilde\Sigma_i,\widetilde\P_i)$
equidistributed with $y_i$ and $z_i$, respectively, such that $z_i\geq y_i$ everywhere on $\widetilde\Omega_i$.

Finally, by taking $(\widetilde\Omega,\widetilde\Sigma,\widetilde\P)$ to be the product space $\prod_{i}\Omega_{i}$
and naturally extending the variables $\widetilde y_{i},\widetilde z_{i}$ to $(\widetilde\Omega,\widetilde\Sigma,\widetilde\P)$,
we obtain the random vectors $\widetilde Y$, $\widetilde Z$ satisfying the required conditions.
\end{proof}

The next lemma provides an actual construction of the required diagonal operator.

\begin{lemma}\label{the heart}
For any $\alpha\in(0,1)$ there is $L=L(\alpha)>0$ with the following property.
Let $(\tau_k)_{k=1}^\infty$ be an increasing non-negative sequence satisfying
$\sum_{k=1}^\infty \tau_{k}2^{-k}<\infty$, and let
$$\widetilde Z:=\sum\limits_{k=0}^\infty\tau_{k+1}\xi^k,$$
where $\xi^k=(\xi^k_1,\xi^k_2,\dots,\xi^k_n)$ and $\xi^k_i$ ($i\leq n$, $k=0,1,\dots$) are jointly independent $0\text{-}1$
random variables with
$\P\{\xi^k_i=1\}=2^{-k}$.
Further, let $\delta\in(0,1]$.
Then there is a random positive contraction $\widetilde D$ taking values in $\Diag_n$ such that
$$\|\widetilde D\widetilde Z\|_1\leq \frac{L}{\delta}\Exp\|\widetilde Z\|_1
=\frac{Ln}{\delta}\sum\limits_{k=0}^\infty\tau_{k+1}2^{-k}\;\;\mbox{everywhere on the probability space},$$
and $\Exp(\det\widetilde D)^{\alpha-1}\le\exp(\delta)$.
\end{lemma}
\begin{proof}
Let $L\ge 2e$ be a number which we will determine later.
Now, for each $k\ge 0$, define random variables
$$\nu_k:=|\{i:\,\xi^k_i\neq 0\}|$$
and
$$\eta_k:=
\begin{cases}
\Bigl(\frac{\delta\nu_k}{L2^{-k}n}\Bigr)^{\nu_k},&\mbox{if }\delta\nu_k\ge L2^{-k}n;\\
1,&\mbox{otherwise.}\end{cases}$$
As building blocks of the contraction $\widetilde{D}$, let us consider random diagonal matrices $D^{(k)}$ with
$$d^{(k)}_{jj}:=\begin{cases}1,&\mbox{if }\xi^k_j=0;\\
\min\bigl(1,\frac{L2^{-k}n}{\delta\nu_k}\bigr),&\mbox{otherwise,}\end{cases}\;\;\;\;j=1,2,\dots,n.$$
Then $\det D^{(k)}={\eta_k}^{-1}$ and $\|D^{(k)}\xi^k\|_1
\le \frac{L2^{-k}n}{\delta}=\frac{L}{\delta}\Exp\|\xi^k\|_1$ (deterministically).
Note that $D^{(k)}$ acts as a dilation on the span of $\{e_i:\,\xi^k_i\neq 0\}$
provided that $\nu_k\ge \frac{L2^{-k}n}{\delta} = \frac{L}{\delta} \Exp\nu_k$, and as an isometry on the orthogonal complement.
We construct the required contraction $\widetilde{D}$ as the product of contractions $D^{(k)}$ by setting
$\widetilde D :=\prod_{k = 0}^{\infty} D^{(k)}$.
Then
$$\|\widetilde D\widetilde Z\|_1\leq
\Bigl\|\sum\limits_{k=0}^\infty \tau_{k+1}D^{(k)}\xi^k\Bigr\|_1\leq
\frac{Ln}{\delta}\sum\limits_{k=0}^\infty \tau_{k+1}2^{-k}=\frac{L}{\delta}\Exp\|\widetilde Z\|_1.$$
Note that
$$\Exp\bigl(\det \widetilde D\bigr)^{\alpha-1}= \Exp\prod_{k=0}^\infty {\eta_k}^{1-\alpha}
=\prod_{k=0}^\infty\Exp{\eta_k}^{1-\alpha}.$$
Next, for every $k\ge 0$ we have
\begin{align*}
\Exp{\eta_k}^{1-\alpha}&\le 1+\sum\limits_{m=\lceil L2^{-k}n/\delta\rceil}^\infty
\Bigl(\frac{\delta m}{L2^{-k}n}\Bigr)^{m-\alpha m}\P\bigl\{\nu_k=m\bigr\}\\
&\le 1+\sum\limits_{m=\lceil L2^{-k}n/\delta\rceil}^\infty
\Bigl(\frac{e\delta}{L}\Bigr)^m\,\Bigl(\frac{L2^{-k}n}{\delta m}\Bigr)^{\alpha m}.
\end{align*}
In particular, for all $k$ such that $L2^{-k}n/\delta\ge 1$, using the relation $L\ge 2e$, we obtain
$$\Exp{\eta_k}^{1-\alpha}\le 1+2\Bigl(\frac{e\delta}{L}\Bigr)^{\lceil L2^{-k}n/\delta\rceil},$$
and for all $k$ satisfying $L2^{-k}n/\delta< 1$, we get
$$\Exp{\eta_k}^{1-\alpha}\le 1+2\frac{e\delta}{L}\bigl(L2^{-k}n/\delta\bigr)^{\alpha}.$$
Now, let us choose $L=L(\alpha)$ sufficiently large so that both
$$\sum\limits_{k:\,L2^{-k}n/\delta\ge 1}2\Bigl(\frac{e\delta}{L}\Bigr)^{\lceil L2^{-k}n/\delta\rceil}\;\;\;\;
\mbox{and}
\;\;\;\;\sum\limits_{k:\,L2^{-k}n/\delta< 1}2\frac{e\delta}{L}\bigl(L2^{-k}n/\delta\bigr)^{\alpha}$$
are less than $\delta/2$.
Then, multiplying the estimates for $\Exp{\eta_k}^{1-\alpha}$, we get
$$\Exp\Bigl(\prod_{k=0}^\infty {\eta_k}\Bigr)^{1-\alpha}\le \exp(\delta),$$
and the result follows.
\end{proof}

\begin{proof}[Proof of Proposition~\ref{vector theorem}]
Fix admissible $\alpha$, $\delta$ and $p$.
Without loss of generality, the distribution of the coordinates of the random vector $X$ is continuous on the real line.
Indeed, otherwise we can replace every coordinate $x_i$ with $|x_i|+u_i$, where $u_1,u_2,\dots,u_n$
are jointly independent with $x_1,x_2,\dots,x_n$ and each $u_i$ is uniformly distributed on $[0,\theta]$
for a very small parameter $\theta>0$ chosen so that $\Exp (|x_i|+u_i)^p\approx\Exp |x_i|^p$.
Then the random diagonal contraction $D$ constructed for the new vector $X':=(|x_i|+u_i)_{i=1}^n$,
will also satisfy the required properties with respect to $X$.

Set $Y:=(|x_1|^p,|x_2|^p,\dots,|x_n|^p)$ and let
$\widetilde Y,\widetilde Z$ be random vectors
on a  space $(\widetilde \Omega,\widetilde\Sigma,\widetilde\P)$ constructed in Lemma~\ref{majorant lemma}
with respect to $Y$. By Lemma~\ref{the heart} and in view of relation \eqref{expectation in terms of tau},
we can find a random positive contraction $\widetilde D$
on $\widetilde \Omega$ taking values in $\Diag_n$
such that for some $L=L(\alpha)>0$ we have
$$\|\widetilde D\widetilde Y\|_1\le \|\widetilde D\widetilde Z\|_1\le \frac{L}{\delta}\Exp\|\widetilde Z\|_1
\le \frac{4L}{\delta}\Exp\|\widetilde Y\|_1\;\;\;\mbox{everywhere on }\widetilde \Omega$$
and
$$\Exp (\det \widetilde D)^{\alpha-1}\le\exp(\delta).$$
In general, the operator $\widetilde D$ is not a function of $\widetilde Y$,
which creates (purely technical) issues in defining corresponding operator on the original space $(\Omega,\Sigma,\P)$.
For completeness, let us describe an elementrary discretization argument resolving the problem:

Let $\{B_z\}$ be a partition of $\R^n_+$ into Borel subsets, indexed over $z=(z_1,z_2,\dots,z_n)\in(\Z\cup\{-\infty\})^n$
and defined by
$$B_z:=\big\{W\in\R^n_+:\,W_i\in (2^{z_i-1},2^{z_i}]\;\;\mbox{for all }i=1,2,\dots,n\big\}$$
(we set $W_i=0$ for $z_i=-\infty$). Further, for every $z$ we let
$$\widetilde\Omega_z:=\big\{\widetilde\omega\in\widetilde\Omega:\,\widetilde Y(\widetilde \omega)\in B_z\big\}$$
and $Q_z:=\widetilde D(\widetilde\Omega_z)=\{M\in \Diag_n:\,M=\widetilde D(\widetilde\omega)\;\;
\mbox{for some }\widetilde\omega\in\widetilde\Omega_z\}$.
For each $z\in(\Z\cup\{-\infty\})^n$ such that $\widetilde\Omega_z$
is non-empty, choose an operator $D_z$ from {\it the closure} of $Q_z$ such that $\det D_z\geq \det M$
for all $M\in Q_z$ (of course, the choice of $D_z$ does not have to be unique).
Otherwise, if $\widetilde\Omega_z$ is empty then we set
$D_z:=\min\big(1,\frac{4L}{\delta \sum_{i=1}^n 2^{z_i}}\Exp\|\widetilde Y\|_1\big){\rm Id}_n$.
Finally, define a function $h:\R^n_+\to\Diag_n$ by setting
$h(W):=D_z$ for all $W\in B_z$ and $z\in(\Z\cup\{-\infty\})^n$. Observe that $h$ is Borel.
Further, by the choice of $D_z$'s, we have $\det h(\widetilde Y)\geq \det \widetilde D$ everywhere on $\widetilde\Omega$,
whence $\Exp (\det h(\widetilde Y))^{\alpha-1}\le\exp(\delta)$.
Next, by the choice of sets $B_z$, we have $\|M(W)\|_1\leq 2\|M'(W')\|_1$ for {\it any} two couples
$(M,W),(M',W')\in Q_z\times B_z$. Together with the conditions on $\widetilde D$ and the definition of $D_z$'s,
this implies $\|D_z(W)\|_1\leq \frac{8L}{\delta}\Exp\|\widetilde Y\|_1$ for all $W\in B_z$, whence
$$\|h(W)\, W\|_1\leq \frac{8L}{\delta}\Exp\|\widetilde Y\|_1\;\;\;\mbox{everywhere on }\R^n_+.$$

Now, taking $T:=h(Y)$, we obtain a random diagonal contraction on $(\Omega,\Sigma,\P)$ such that
$$\|T^{1/p}X\|_p^p=\|T Y\|_1\le \frac{8L}{\delta}\Exp\|X\|_p^p\;\;\;\mbox{everywhere on $\Omega$}$$
and $\Exp (\det T)^{\alpha-1}\le\exp(\delta)$.
Finally, setting $D:=T^{1/p}$, we get the required operator.
\end{proof}

\bigskip

The above statement can be ``tensorized''. In what follows, we are interested only in the case $p=2$ and $\alpha=1/2$.
\begin{prop}\label{bounded rows and det}
There is a universal constant $C>0$
with the following property.
Let $A=(a_{ij})$ be an $n\times n$ random matrix satisfying \eqref{probab model},
and let $\delta\in(0,1]$.
Then there is a random positive contraction $D$ taking values in $\Diag_n$ such that
the Euclidean norms of the rows of $AD$ are uniformly bounded by
$\frac{C}{\sqrt{\delta}}\sqrt{n}$ everywhere on the probability space, and
$$\Exp\det D^{-1}\le \exp(\delta n).$$
\end{prop}
\begin{proof}
Indeed, for any $i=1,2,\dots,n$, let $D_i$ be the positive contraction defined with respect to the $i$-th row of $A$
using Proposition~\ref{vector theorem}
(with parameters $\alpha=1/2$, $p=2$), so that $D_1,D_2,\dots,D_n$ are jointly independent.
Then the product of these contractions $D:=\prod_{i=1}^n D_i$ satisfies the required conditions. 
\end{proof}

\begin{rem} \label{expectation remark}
It is not difficult to see that for any positive contraction $M\in\Diag_n$
there is an element $\widetilde M\in\DiagDyadic_n$ such that $\widetilde M\le \sqrt{2}M$ and
$\det \widetilde M^{-1}\le \det M^{-2}$.
Indeed, this follows easily from the fact that for any number $t\in(0,1]$ there is $\widetilde t\in\{1\}\cup\bigl\{2^{-2^k}\bigr\}_{k=0}^\infty$
with $t^2\le\widetilde t\le \sqrt{2}t$ (the constant $\sqrt{2}$ on the right-hand side is achieved for $t= \sqrt{2}/2-o(1)$).
Hence, the above statement implies that, given a matrix $A$ satisfying \eqref{probab model} and a number $\delta >0$,
one can construct a random contraction $\widetilde D$ taking values in $\DiagDyadic_n$ such that each row of $A\widetilde D$ has Euclidean norm
at most $\frac{C}{\sqrt{\delta}}\sqrt{n}$
(for some universal constant $C>0$), and $\Exp\det \widetilde D^{-1/2}\le \exp(\delta n).$
\end{rem}

\section{Coverings of random ellipsoids}\label{parallel section}

The main result of the section is
\begin{theor}\label{parallelepiped norm estimate}
Let $\delta\in(0,1]$ and let $A=(a_{ij})$ be an $n \times n$ random matrix satisfying \eqref{probab model}.
Then
$$\P\Bigl\{\exists D\in \DiagDyadic_n:
\det D\ge\exp(-\delta n)\mbox{ and }\|AD\|_{\infty\to 2}
\le \frac{C_{\text{\tiny\ref{parallelepiped norm estimate}}}}{\sqrt{\delta}}n\Bigr\}\ge 1-4\exp(-\delta n/8),$$
where $C_{\text{\tiny\ref{parallelepiped norm estimate}}}>0$ is a universal constant.
\end{theor}

\begin{rem}
The above theorem can be seen as a way to ``regularize'' the random matrix $A$ by reducing
its norm while preserving its ``structure''. In this connection, let us mention work \cite{LV} where
a very general problem of regularizing random matrices was discussed (see \cite[Section~5.4]{LV}).
\end{rem}

As we have mentioned in the introduction, Theorem~A follows almost immediately from
the above statement; we give the proof of Theorem~A at the very end of the section.
The section is organized as follows.
First, we use $\widetilde{D}$ constructed in Remark~\ref{expectation remark}
to verify Theorem~\ref{parallelepiped norm estimate} under an additional assumption
that the entries of $A$ are symmetrically distributed
(see Proposition~\ref{parallelepiped norm estimate sym}).
Then, we will apply a symmetrization procedure to prove
Theorem~\ref{parallelepiped norm estimate} in full generality.

\bigskip

A random variable $\xi$ is {\it subgaussian} if there exists a number $K>0$ such that
\begin{equation}\label{subgaussian equation}
\P\{|\xi| > t\} \le 2\exp\bigl(-t^2/K^2\bigr),\;\; t>0.
\end{equation}
To put an emphasis on the value of $K$, we will sometimes call $\xi$ $K$-subgaussian.
We note that the smallest value of $K$ satisfying \eqref{subgaussian equation} is equivalent
to {\it the subgaussian norm} of $\xi$ (see, for example, \cite[Lemma~5.5]{Vershynin});
however, the latter notion is less convenient for us and will not be used in this paper.

The next lemma is equivalent to a standard Khintchine--type inequality (see, for example, \cite{Hoeffding}).
\begin{lemma}\label{Khintchine} Let $r_1,r_2,\dots,r_n$
be independent Rademacher random variables.
Then for any vector $y\in S^{n-1}$ the random variable $\sum_{i = 1}^n y_i r_i$
is $C_{\text{\tiny\ref{Khintchine}}}$-subgaussian, where
$C_{\text{\tiny\ref{Khintchine}}} > 0$ is a universal constant.
\end{lemma}

The sum of squares of subgaussian variables has good concentration properties;
the bound below follows from a standard ``Laplace transform'' argument (see, for
example, \cite[Corollary~5.17]{Vershynin}):
\begin{lemma}\label{sum of subgaussians}
For any $T>0$ there is $L_{\text{\tiny\ref{sum of subgaussians}}}>0$ depending on $T$ with the following property:
Let $\xi_1,\xi_2,\dots,\xi_n$ be independent centered $1$-subgaussian random variables.
Then
$$\P\Bigl\{\sum\limits_{i=1}^n\xi_i^2>L_{\text{\tiny\ref{sum of subgaussians}}}n\Bigr\}\le\exp(-Tn).$$
\end{lemma}

The next proposition implies that for a random matrix $A$ satisfying \eqref{probab model}
with symmetrically distributed entries and the operator $\widetilde D$ from
Remark~\ref{expectation remark}, the norm $\|A\widetilde D\|_{\infty\to 2}$ can be efficiently bounded from above
as long as $\widetilde D$ is a Borel function of $|A|$ (here and further in the text, given a matrix $B=(b_{ij})$, 
by $|B|$ we shall denote the matrix $(|b_{ij}|)$).

\begin{prop}\label{parallelepipeds sym}
Let $K>0$ and let $A$ be an $n \times n$ random matrix satisfying \eqref{probab model}, with symmetrically distributed entries.
Further, let $\F\subset \Diag_n$ be any countable subset. Denote by $\Event$ the event
$$
\Event:=\bigl\{\exists D \in \F:\mbox{ all rows of }AD\mbox{ have Euclidean norms at most } K\sqrt{n}\bigr\}.
$$
Then
$$\P\{\exists D\in \F:\,\|AD\|_{\infty\to 2}\le CK n\}\ge \P(\Event) - \exp(-n),$$
where $C>0$ is a universal constant.
\end{prop}
\begin{proof}
Fix any admissible $K$ and $\F$. Clearly, for any $n\times n$ matrix $B$ and a diagonal matrix $D$,
the Euclidean norms of rows of $BD$ and $|B|D$ are the same.
Hence, we may assume that there is a Borel function $f:\R_{+}^{n\times n}\to \F$ such that
$$\Event=\bigl\{\mbox{all rows of }A\,f(|A|)\mbox{ have norms at most }K\sqrt{n}\bigr\}.$$
For any $D\in\F$, let
$$\Event_D:=\Event\cap\{f(|A|)=D\}.$$
Without loss of generality, $\P(\Event_D)>0$ for any $D\in\F$.

Next, as the unit cube $[-1,1]^n$ is the convex hull of its vertices $V=\{-1,1\}^n$, we have
\begin{equation}\label{infty to 2 representation}
\|Af(|A|)\|_{\infty\to 2}=\sup\limits_{y \in B^n_{\infty}} \|Af(|A|) y\| =   \sup\limits_{v \in V} \|Af(|A|) v\|.
\end{equation}

Note that, given event $\Event_D$, the entries of $A f(|A|)=AD$ are symmetrically distributed,
so the distribution of $ADv$ given $\Event_D$ is the same for any vertex $v\in V$. Fix a vertex $v$.

Observe that for any $t>0$ we have
\begin{equation}\label{bound via sup}
\P_{\Event_D}\{\|ADv\|>t\}\le\sup\limits_{B}\P\{\|\widetilde BDv\|>t\},
\end{equation}
where by $\P_{\Event_D}$ we denote the conditional probability given $\Event_D$
and the supremum is taken over all matrices $B=(b_{ij})$ such that the rows of $BD$ have
Euclidean norms at most $K\sqrt{n}$,
and $\widetilde B=(r_{ij}b_{ij})$, with $r_{ij}$ being jointly independent Rademacher ($\pm 1$) variables.
Fix any admissible $B=(b_{ij})$. 

Then the variables $\langle\widetilde B Dv,e_i\rangle$, $i=1,2,\dots,n$,
are jointly independent and, in view of Lemma~\ref{Khintchine} and the choice of $B$, each variable
$K^{-1}n^{-1/2}\langle\widetilde B Dv,e_i\rangle$ is $C_{\text{\tiny\ref{Khintchine}}}$-subgaussian.
By Lemma~\ref{sum of subgaussians}, there is a universal constant
$C>0$ such that
$$\P\{\|\widetilde BDv\|>CKn\} =
\P\Bigl\{\frac{1}{n} \sum_{i=1}^n \langle\widetilde B Dv,e_i\rangle^2 > (CK)^2 n \Bigr\} \le \exp\bigl(-(1 + \ln 2) n\bigr).$$
Then, taking a union bound over $2^n$ vertices of the unit cube and
using \eqref{bound via sup} and \eqref{infty to 2 representation}, we get an estimate
$$\P_{\Event_D}\{\|AD\|_{\infty\to 2}>CKn\}\le 2^n \cdot \sup\limits_{B}\P\{\|\widetilde BDv\|>CKn\} \le \exp(-n).$$
Finally, clearly
$$\P\{\|AD\|_{\infty\to 2}>CKn\}\le \P(\Event^c)+\sum\limits_D\P_{\Event_D}\{\|AD\|_{\infty\to 2}>CKn\}\P(\Event_D)
\le \P(\Event^c)+\exp(-n),$$
and the result follows.
\end{proof}

\begin{prop}\label{parallelepiped norm estimate sym}
Let $\delta\in(0,1]$ and let $A=(a_{ij})$ be an $n \times n$ random matrix satisfying \eqref{probab model}, with symmetrically distributed entries.
Then
$$\P\bigl\{\exists D\in \DiagDyadic_n:\det D\ge\exp(-\delta n)\mbox{ and }\|AD\|_{\infty\to 2}\le
\frac{C_{\text{\tiny\ref{parallelepiped norm estimate sym}}}}{\sqrt{\delta}}n\bigr\}
\ge 1-2\exp(-\delta n/4).$$
\end{prop}
\begin{proof}
Fix any $\delta\in(0,1]$.
In view of Remark~\ref{expectation remark}, there is
a random contraction $D$ taking values in $\DiagDyadic_n$ such that each row of $AD$ has the Euclidean norm at most
$\frac{C}{\sqrt{\delta}}\sqrt{n}$ and
$\Exp\det D^{-1/2}\le \exp(\delta n/4)$. Denote by $\Event$ the event
$$\Event:=\bigl\{\det D\ge\exp(-\delta n)\bigr\}.$$
In view of the conditions on $D$ and Markov's inequality, we have
$$\P(\Event)\ge 1-\exp(-\delta n/4).$$
Hence, by Proposition~\ref{parallelepipeds sym}, taking $\F$ to be the set of all contractions from $\DiagDyadic_n$ having determinant at least
$\exp(-\delta n)$, we obtain
\begin{align*}
\P\bigl\{\exists D\in \DiagDyadic_n:\det D\ge\exp(-\delta n)
\mbox{ and }\|AD\|_{\infty\to 2}
\le\frac{C_{\text{\tiny\ref{parallelepiped norm estimate sym}}}}{\sqrt{\delta}}n\bigr\}
\ge 1-\exp(-\delta n/4)-\exp(-n)
\end{align*}
for a universal constant $C_{\text{\tiny\ref{parallelepiped norm estimate sym}}}>0$.
\end{proof}

For the next lemma we will need the following definition (essentially taken from \cite{MS}).
Let $S$ be a finite set and $d$ be a pseudometric on $S$. We say that $(S,d)$ is {\it of length at most $\ell$}
(for some $\ell>0$) if there is $n\in\N$, positive numbers $b_1,b_2,\dots,b_n$ with $\|(b_1,b_2,\dots,b_n)\|\le\ell$
and a sequence $(S_k)_{k=0}^n$ of partitions of $S$ such that
\begin{enumerate}
\item $S_0=\{S\}$;
\item $S_n=\{\{s\}\}_{s\in S}$;
\item $S_k$ is a refinement of $S_{k-1}$ for all $k=1,2,\dots,n$;
\item For each $k\in\{1,2,\dots,n\}$ and any $Q,Q'\in S_k$ such that $Q\cup Q'$ is a subset of an element of $S_{k-1}$,
there is a one-to-one mapping $\phi:Q\to Q'$ such that $d(s,\phi(s))\le b_k$ for all $s\in Q$.
\end{enumerate}
In particular, the above conditions on $S_k$ imply that all elements of $S_k$ have the same cardinality.
\begin{theor}[see {\cite[Theorem~7.8]{MS}}]\label{MS theorem}
Let $(S,d)$ be a finite pseudometric space of length at most $\ell$ and let $\mu$ be the normalized counting measure on $S$.
Then for any function $f:S\to\R$ satisfying $|f(s)-f(s')|\le d(s,s')$ ($s,s'\in S$) and all $t>0$ we have
$$\mu\Bigl\{\Bigl|f-\int f\,d\mu\Bigr|\ge t \Bigr\}\le 2\exp\Bigl(-\frac{t^2}{4\ell^2}\Bigr).$$
\end{theor}
\begin{rem}
In \cite{MS}, the above theorem is formulated for metric spaces. It is easy to see that passing to pseudometrics
does not change the picture.
\end{rem}

Denote by $\Pi_n$ the set of permutations of $[n]:=\{1,2,\dots,n\}$.

\begin{lemma}\label{cor of azuma}
Let $y=(y_1,y_2,\dots,y_n)$ be a non-zero vector
and $v = (v_1, v_2, \dots, v_n)$ be a vertex of the cube $[-1,1]^n$.
Further, let $\mu$ be the normalized counting measure on $\Pi_n$. Define
a function $f:\Pi_n\to\R$ as
$$f(p):=\sum\limits_{j=1}^n v_{p(j)}y_j,\;\;\;p\in\Pi_n.$$
Then
$$\mu\Bigl\{\Bigl|f-\int f\,d\mu\Bigr|\ge t\Bigr\}\le 2\exp\Bigl(-\frac{t^2}{64\|y\|^2}\Bigr),\;\; t>0.$$
\end{lemma}
\begin{proof}
Without loss of generality, we can assume that $|y_j|\ge |y_{j+1}|$ ($j=1,2,\dots,n-1$).
Define a pseudometric $d$ on $\Pi_n$:
for any $p,q\in\Pi_n$ let
\begin{equation*}\label{condition on f}
d(p,q):=|f(p)-f(q)|.
\end{equation*}
Further, we define a sequence of partitions $(\Pi_{n,k})_{k=0}^n$ of $\Pi_n$:
let $\Pi_{n,0}:=\{\Pi_n\}$ and for each $k=1,2,\dots,n$, let
$\Pi_{n,k}$ consist of all subsets of $\Pi_n$ of the form
$$\{p\in\Pi_n:\;p(1)=i_1,p(2)=i_2,\dots,p(k)=i_k\}$$
for all $\{i_1,i_2,\dots,i_k\}\subset [n]$.

Now, let $k\in\{1,2,\dots,n\}$ and let $Q,Q'\in\Pi_{n,k}$ be such that $Q\cup Q'$ is a subset of an element of $\Pi_{n, k-1}$.
Note that there are numbers $i_1,i_2,\dots,i_k$, $i_k'$ such that
$p(j)=i_j$ for all $j<k$ and $p\in Q\cup Q'$; $p(k)= i_k$ for all $p\in Q$ and $p(k)=i_k'$ for all $p\in Q'$.
Define a one-to-one mapping $\phi:Q\to Q'$ by
$$\phi(p)(j):=p(j)\mbox{ for }j\neq k,p^{-1}(i_k');\;\;\phi(p)(k):=i_k';\;\;\phi(p)(p^{-1}(i_k')):=i_k.$$
For any $p\in Q$, we have
$$d(p,\phi(p))\le 2|y_k|+2|y_{p^{-1}(i_k')}|\le 4|y_k|,$$
with the last inequality due to the fact that $p^{-1}(i_k')\ge k$.
Thus, the space $(\Pi_n,d)$ is of length at most $4\|y\|$.
Applying Theorem~\ref{MS theorem}, we get the result.
\end{proof}

The next statement shall be used in a symmetrization argument within the proof of Theorem~\ref{parallelepiped norm estimate};
we think it may be of interest in itself.

\begin{prop}\label{permutation model}
Let $B=(b_{ij})$ be a non-random $n\times n$ matrix such that the Euclidean norm of every row is at most $\sqrt{n}$
and such that
$$\Bigl|\sum\limits_{j=1}^n b_{ij}\Bigr|\le\sqrt{n},\;\;i=1,2,\dots,n.$$
Further, let $\pi_i$ ($i=1,2,\dots,n$) be independent random permutations uniformly distributed on $\Pi_n$, and denote by
$\widetilde B=(\widetilde b_{ij})$ the random $n\times n$ matrix with entries defined by
$$\widetilde b_{ij}:=b_{i,\pi_i(j)}.$$
Then
$$\P\{\|\widetilde B\|_{\infty\to 2}\le C_{\text{\tiny\ref{permutation model}}}n\}\ge 1-\exp(-n)$$
for a universal constant $C_{\text{\tiny\ref{permutation model}}}>0$.
\end{prop}
\begin{proof}
We will show that for any $v\in\{-1,1\}^n$ we have
$$\P\{\|\widetilde Bv\|> C_{\text{\tiny\ref{permutation model}}}n\}\le \exp(-n-n\ln 2)$$
for a sufficiently large universal constant $C_{\text{\tiny\ref{permutation model}}}$
and then take the union bound over the vertices of the cube.

Fix any $v=(v_1,v_2,\dots,v_n)\in\{-1,1\}^n$ and let $m$ be the number
of ones in $(v_1,\dots,v_n)$.
Clearly, the random variables $\langle \widetilde Bv,e_i\rangle$ ($i=1,2,\dots,n$)
are independent. Next, for a fixed $i$, the distribution of $\langle \widetilde Bv,e_i\rangle$ coincides with
that of the variable $\xi_i:=\sum_{j=1}^n v_{\pi_i(j)}b_{ij}$.
By Lemma~\ref{cor of azuma} and in view of the condition on the rows of $B$, we have
$$\P\bigl\{|\xi_i-\Exp\xi_i|>\tau\bigr\}\le 2\exp\Bigl(-\frac{\tau^2}{64n}\Bigr),\;\;\tau>0.$$
Hence, the variables $n^{-1/2}(\xi_i-\Exp\xi_i)$ ($i=1,2,\dots,n$) are $C$-subgaussian for a universal constant $C>0$.
In view of Lemma~\ref{sum of subgaussians}, we get that
\begin{equation}\label{xi minus exp}
\P\Bigl\{\sum\limits_{i=1}^n (\xi_i-\Exp\xi_i)^2>\widetilde Cn^2\Bigr\}\le\exp(-n-n\ln 2)
\end{equation}
for some constant $\widetilde C>0$. Finally, observe that
$$\sum\limits_{i=1}^n \xi_i^2\le 2\sum\limits_{i=1}^n (\xi_i-\Exp\xi_i)^2+2\sum\limits_{i=1}^n\bigl(\Exp\xi_i\bigr)^2
\;\;\;\mbox{(deterministically)},$$
so, applying the estimate
$$\bigl|\Exp\xi_i\bigr|=\Bigl|\frac{2m-n}{n}\sum\limits_{j=1}^n b_{ij}\Bigr|\le \sqrt{n}$$
and \eqref{xi minus exp}, we obtain
$$\P\bigl\{\|\widetilde Bv\|^2>(2\widetilde C+2)n^2\bigr\}
=\P\Bigl\{\sum\limits_{i=1}^n \xi_i^2>(2\widetilde C+2)n^2\Bigr\}\le\exp(-n-n\ln 2).$$
\end{proof}

\begin{proof}[Proof of Theorem~\ref{parallelepiped norm estimate}]
Let $\widetilde A$ be an independent copy of $A$.
Obviously
$$\Exp\Bigl(\sum\limits_{j=1}^n \widetilde a_{ij}\Bigr)^2=\Exp\sum\limits_{j=1}^n \widetilde a_{ij}^2=n$$
for every $i=1,2,\dots,n$. Then, in view of Markov's inequality, each row of $\widetilde A$ satisfies
$$\Bigl|\sum\limits_{j=1}^n \widetilde a_{ij}\Bigr|\le \sqrt{\frac{32n}{\delta}}\;\;\mbox{and}\;\;
\sum\limits_{j=1}^n {\widetilde a_{ij}}^2\le \frac{32n}{\delta}$$
with probability at least $1-\delta/16> \exp(-\delta/8)$.
Denote by $\widetilde \Event$ the event
$$\widetilde\Event:=\Bigl\{\Bigl|\sum\limits_{j=1}^n \widetilde a_{ij}\Bigr|\le \sqrt{\frac{32n}{\delta}}\;\;\mbox{and}\;\;
\sum\limits_{j=1}^n {\widetilde a_{ij}}^2\le \frac{32n}{\delta}\;\;\mbox{for all }i=1,2,\dots,n\Bigr\}.$$
In view of the above, $\P(\widetilde\Event)\ge\exp(-\delta n/8)$.
Let $\pi_1,\pi_2,\dots,\pi_n$ be random permutations uniformly distributed on $\Pi_n$ and jointly
independent with $\widetilde A$, and denote by $\widetilde B=(\widetilde b_{ij})$ the random matrix with
entries $\widetilde b_{ij}:=\widetilde a_{i,\pi_i(j)}$ ($i,j\leq n$).
Then Proposition~\ref{permutation model} yields
$$\P\bigg\{\|\widetilde B\|_{\infty\to 2}\le C_{\text{\tiny\ref{permutation model}}}\sqrt{n}\max\limits_{i\leq n}
\Big(\sum_{j=1}^n {\widetilde a_{ij}}^2\Big)^{1/2}\;\mid\;\widetilde A\bigg\}\ge 1-\exp(-n),$$
whence, in particular,
$$\P\big\{\|\widetilde B\|_{\infty\to 2}\le C_{\text{\tiny\ref{permutation model}}}\sqrt{32/\delta}\,n\,|\,\widetilde\Event\big\}\ge 1-\exp(-n).$$
But $\widetilde B$ is equidistributed with $\widetilde A$ given $\widetilde\Event$, so that
$$\P\big\{\|\widetilde A\|_{\infty\to 2}\le C_{\text{\tiny\ref{permutation model}}}\sqrt{32/\delta}\,n\,|\,\widetilde\Event\big\}\ge 1-\exp(-n).$$
Clearly, $\|\widetilde AD\|_{\infty\to 2}\le\|\widetilde A\|_{\infty\to 2}$ for any contraction $D\in\Diag_n$ (deterministically),
so we obtain for the event
$\Event_1:=\big\{\|\widetilde A D\|_{\infty\to 2}\le C_{\text{\tiny\ref{permutation model}}}\sqrt{32/\delta}\,
n\;\;\mbox{for all }D\in\Diag_n\big\}$:
$$\P(\Event_1)\ge(1-\exp(-n))\P(\widetilde\Event)\ge\frac{1}{2}\exp(-\delta n/8).$$
Next, the matrix $2^{-1/2}(A-\widetilde A)$ has symmetrically distributed entries, and satisfies conditions
of Proposition~\ref{parallelepiped norm estimate sym}. Hence,
\begin{align*}
\P&\Bigl\{\|(A-\widetilde A)D\|_{\infty\to 2}\leq C_{\text{\tiny\ref{parallelepiped norm estimate sym}}}\sqrt{2/\delta}\,n
\;\;\mbox{for some }D\in\DiagDyadic_n\mbox{ with }\det D\ge\exp(-\delta n)\Bigr\}\\
&\geq 1- 2\exp(-\delta n/4).
\end{align*}
Conditioning on $\Event_1$, we get
\begin{align*}
\P&\Bigl\{\|(A-\widetilde A)D\|_{\infty\to 2}\leq C_{\text{\tiny\ref{parallelepiped norm estimate sym}}}\sqrt{2/\delta}\,n
\;\;\mbox{for some }D\in\DiagDyadic_n\mbox{ with }\det D\ge\exp(-\delta n)\,|\,\Event_1\Bigr\}\\
&\geq 1-\frac{2\exp(-\delta n/4)}{\P(\Event_1)}\\
&\geq 1- 4\exp(-\delta n/8).
\end{align*}
Note that, given $\Event_1$, we have $\|AD\|_{\infty\to 2}\leq \|(A-\widetilde A)D\|_{\infty\to 2}
+C_{\text{\tiny\ref{permutation model}}}\sqrt{32/\delta}\,n$
{\it for all} contractions $D\in\Diag_n$. Combining this with the last formula, we obtain
\begin{align*}
\P\Bigl\{&\|AD\|_{\infty\to 2}\leq C_{\text{\tiny\ref{parallelepiped norm estimate sym}}}\sqrt{2/\delta}\,n
+C_{\text{\tiny\ref{permutation model}}}\sqrt{32/\delta}\,n\\
&\mbox{for some }D\in\DiagDyadic_n\mbox{ with }\det D\ge\exp(-\delta n)\,|\,\Event_1\Bigr\}
\geq 1- 4\exp(-\delta n/8).
\end{align*}
Finally, since $A$ is independent from $\Event_1$, the conditioning in the last estimate can be dropped,
and we obtain the statement.
\end{proof}

To complete the proof of Theorem~A, we will need two more technical lemmas:
\begin{lemma}\label{quantity}
For any $\delta\in(0,1/2]$ and all $n\in\N$ we have
$$\bigl|\bigl\{D\in\DiagDyadic_n:\,\det D\ge\exp(-\delta n)\bigr\}\bigr|\le \Bigl(\frac{2e}{\delta}\Bigr)^{4\delta n}.$$
\end{lemma}
\begin{proof}
Denote ${\mathcal S}:=\{D\in\DiagDyadic_n:\,\det D\ge\exp(-\delta n)\bigr\}$.
Note that for any matrix $D\in {\mathcal S}$
and for any $k\ge 0$, the number of diagonal elements of $D$ equal to $2^{-2^k}$
is less than $2^{-k+1}\delta n$. Hence, the cardinality of ${\mathcal S}$ can be estimated as
$$|{\mathcal S}|\le \prod_{k=0}^\infty {n\choose [2^{-k+1}\delta n]}\le
\prod_{k=0}^\infty\Bigl(\frac{e}{\delta}\Bigr)^{2^{-k+1}\delta n}2^{k2^{-k+1}\delta n}
=\Bigl(\frac{e}{\delta}\Bigr)^{4\delta n}2^{4\delta n}
=\Bigl(\frac{2e}{\delta}\Bigr)^{4\delta n}.$$
\end{proof}

\begin{lemma}\label{cube extended lemma}
For any $n\in\N$ and $K\in[2,2\sqrt{n}]$, the unit Euclidean ball $B_2^n$ can be covered
by at most $(2eK^2)^{8n/K^2}$ translates of the cube $\frac{K}{\sqrt{n}}B_\infty^n$.
\end{lemma}
\begin{proof}
First, note that for any $y\in B_2^n$ we have
$$\Bigl|\Bigl\{i\le n:\,|y_i|\ge \frac{K}{2\sqrt{n}}\Bigr\}\Bigr|\le\frac{4n}{K^2}.$$
Hence, it is sufficient to show that the set $|\{y\in B_2^n:\,|\supp(y)|\le\frac{4n}{K^2}\}|$
can be covered by at most $(2eK^2)^{8n/K^2}$ translates of $\frac{K}{2\sqrt{n}}B_\infty^n$.
A simple volumetric argument, together with an estimate $\Vol(B_2^m)\le \bigl(\frac{2\pi e}{m}\bigr)^{m/2}$,
implies that $B_2^m$ can be covered by at most $7^m$ translates of $\frac{1}{\sqrt{m}}B_\infty^m$
(for any $m\in\N$). As a consequence, we obtain a covering of
$B_2^{\lceil 4n/K^2\rceil}$ by at most $7^{\lceil 4n/K^2\rceil}$ translates of $\frac{K}{2\sqrt{n}}B_\infty^n$.
Finally, the cardinality of the optimal covering of $|\{y\in B_2^n:\,|\supp(y)|\le\frac{4n}{K^2}\}|$
can be estimated from above by
$${n\choose \lceil 4n/K^2\rceil}7^{\lceil 4n/K^2\rceil}\le\bigl(2eK^2\bigr)^{8n/K^2}.$$
\end{proof}

\begin{proof}[Proof of Theorem~A]
Let $\delta\in(0,1/4]$ and $n \ge \frac{1}{4\delta}$.
First, applying Lemma~\ref{cube extended lemma} with $K=1/\sqrt{\delta}$,
we see that $B^n_2$ can be covered by $(2e/\delta)^{8n\delta}$ translates
of the dilated cube $\frac{1}{\sqrt{n \delta}}B^n_{\infty}$. Let
$$
  {\mathcal Q} = \{D\in\DiagDyadic_n:\,\det D\ge\exp(-\delta n)\}.
$$
Then, in view of Lemma~\ref{quantity}, we get that $B_\infty^n$ can be covered by at most
$(2e/\delta)^{4\delta n}\exp(\delta n)$ parallelepipeds in such a way that
for any $y\in B_\infty^n$ and $D\in{\mathcal Q}$, $y$ is covered by a translate of $D(B_\infty^n)$.
Combining the two coverings, we get a collection $\mathcal C$ of parallelepipeds covering $B_2^n$
such that
$$|\mathcal C|\leq (2e/\delta)^{4\delta n}\exp(\delta n)\cdot(2e/\delta)^{8n\delta} = \exp(\delta n + 12 n \delta \ln\frac{2e}{\delta}),$$
and for any $y\in B_2^n$ and $D\in{\mathcal Q}$, the set $\mathcal C$ contains
a translate of $D(\frac{1}{\sqrt{n \delta}}B_\infty^n)$ covering $y$.
Finally, applying Theorem~\ref{parallelepiped norm estimate}, we get that with probability
at least $1-4\exp(-\delta n/8)$ for some $D\in{\mathcal Q}$ we have
$AD(B^n_{\infty}) \subset
\frac{C_{\text{\tiny\ref{parallelepiped norm estimate}}}n}{\sqrt{\delta}} B^n_2$, implying
\begin{align*}
\P&\Big\{\forall\; x\in B_2^n\;\;\exists \;P\in\mathcal C\;\;\mbox{such that}\;\;x\in P\mbox{ and }
A(P)\subset Ax+\frac{2\cdot C_{\text{\tiny\ref{parallelepiped norm estimate}}}\sqrt{n}}{\delta}B_2^n\Big\}\\
&\geq 1-4\exp(-\delta n/8).
\end{align*}
(the multiple ``$2$'' in the last formula
appears because the translation $-Ax+A(P)$ is not origin-symmetric in general).
\end{proof}

\begin{proof}[Proof of Corollary~A]
Fix $n$ and $\delta$, and let $\mathcal C$ be the collection of parallelepipeds defined in Theorem~A.
For each $P\in\mathcal C$, choose a point $y_p\in P\cap B_2^n$, and let $\Net:=\{y_P:\,P\in\mathcal C\}$.
Then, clearly,
$$
|\Net| = |\mathcal C|\leq \exp(\delta n + 12 n \delta \ln\frac{2e}{\delta}),
$$ 
and with probability at least $1-4\exp(-\delta n/8)$ 
for every $x\in B_2^n$ there is $y=y(x)\in\Net$ with $-Ax+Ay\in \frac{C\sqrt{n}}{\delta}B_2^n$.
In short,
$$\P\Big\{A(B_2^n)
\subset \bigcup_{y\in A(\Net)}\bigl(y
+\frac{C\sqrt{n}}{\delta}B_2^n\bigr)\Big\}\geq 1-4\exp(-\delta n/8).$$
\end{proof}

\section{The smallest singular value --- Preliminaries}\label{sec prelim}

As we already mentioned in the introduction, the proof of Theorem~B heavily
relies on results obtained by Rudelson and Vershynin in papers \cite{RV square} and \cite{RV rectangular}.
In this section, we will state several intermediate results from those papers that we will need
in Section~\ref{ssv section} to complete our proof.

A crucial step in the proof of \cite[Theorem~1.2]{RV square} is a decomposition of the unit sphere into
sets of ``compressible'' and ``incompressible'' vectors.
\begin{definition}[Sparse, compressible and incompressible vectors]
Fix parameters $\theta, \rho \in (0, 1)$. A vector $x \in \R^n$ is called {\it $\theta n$-sparse} if $|\supp(x)| \le \theta n$.
A vector $x \in S^{n-1}$ is called {\it compressible} if $x$ is within Euclidean distance $\rho$
from the set of all $\theta n$-sparse vectors. Otherwise, $x$ will be called {\it incompressible}.
The set of all compressible unit vectors will be denoted by $\Comp_n(\theta,\rho)$, and the set of incompressible
 vectors --- by $\Incomp_n(\theta,\rho)$. Sometimes, when the dimension $n$ or the parameters $\theta,\rho$ are
clear from the context, we will simply write $\Comp$, $\Incomp$ to denote the sets.
\end{definition}
\begin{rem}
A similar decomposition of the unit sphere was already introduced in an earlier paper \cite{LPRT} for the purpose
of bounding the smallest singular value of rectangular matrices.
\end{rem}

Obviously, for any $\varepsilon>0$ we have
$$
\P\{s_{n}(A) < \varepsilon n^{-1/2}\}  \le
\P\bigl\{\inf_{y \in \Comp} \|Ay\| < \varepsilon n^{-1/2}\bigr\}
+  \P\bigl\{\inf_{y \in \Incomp} \|Ay\| < \varepsilon n^{-1/2}\bigr\}.
$$
Treatment of the compressible vectors is simpler due to the fact the the set $\Comp$ is ``small'';
we will deal with this set in the first part of Section~\ref{ssv section}. Let us remark that, unlike in the subgaussian
result of \cite{RV square},
where an estimate for compressible vectors follows almost directly from an analogue of Lemma~\ref{tensorization lemma}
(see below) together
with a standard covering argument, in our case we will still need to use additional results (proved in Section~\ref{parallel section})
as the norm $\|A\|_{2\to 2}$ may be ``too large''.
We will need the following simple lemma:
\begin{lemma}\label{estimating comp}
For any $\theta,\rho\in(0,1]$ the set $\Comp=\Comp_n(\theta,\rho)$ admits a Euclidean $3\rho$-net $\Net\subset \Comp$
of cardinality $|\Net|\le (e/\theta)^{\theta n}\bigl(\frac{5}{\rho}\bigr)^{\theta n}$.
\end{lemma}
\begin{proof}
Note that the definition of $\Comp$ implies that for any $y\in\Comp$ there is $y'\in S^{n-1}$
such that $|\supp(y')|\le \theta n$ and $\|y-y'\|\le 2\rho$.
Hence, it is enough to show that one can find a Euclidean $\rho$-net $\Net$ on the set of $\theta n$-sparse unit vectors,
with the required estimate on $|\Net|$.
This follows from a standard estimate on the cardinality of an optimal $\rho$-net on $S^{\lfloor\theta n\rfloor-1}$, together
with a bound for the binomial coefficient ${n\choose \lfloor\theta \rfloor}$.
\end{proof}

Incompressible vectors have the important property that a significant portion of their coordinates are of order $n^{-1/2}$.
In paper \cite{RV square}, this property was referred to as ``incompressible vectors are spread''.
For reader's convenience, we provide a proof of this fact below (let us note once again that analogous concepts were already considered
in \cite{LPRT}).
\begin{lemma}[{\cite[Lemma~3.4]{RV square}}]\label{incompressible are spread}
For any $\theta,\rho\in(0,1)$
and for any vector $x\in\Incomp_n(\theta,\rho)$
there is a subset of indices $\sigma(x)\subset\{1,2,\dots,n\}$
of cardinality at least $\frac{1}{2}\rho^2 \theta n$
such that for all $i \in \sigma(x)$ we have
$$
\frac{\rho}{\sqrt{2n}} \le x_i \le \frac{1}{\sqrt{\theta n}}.
$$
\end{lemma}
\begin{proof}
For every subset $I\subset\{1,2,\dots,n\}$, let $P_I$ be the coordinate projection onto the span of $\{e_i:\,i\in I\}$.
Let $\sigma=\sigma(x) := \sigma_1 \cap \sigma_2$, where
$$ 
\sigma_1 = \Big\{i\leq n:\, |x_i| \le \frac{1}{\sqrt{\theta n}}\Big\}, \quad \quad
\sigma_2 = \Big\{i\leq n:\, |x_i| \ge \frac{\rho}{\sqrt{2 n}}\Big\}.
$$
Since $\|x\| = 1$, we have $|\sigma_1^c| \le \theta n$,
and $P_{\sigma_1^c}(x)$ is a $\theta n$-sparse vector.
Then the condition that $x$ is incompressible implies
$\|P_{\sigma_1}(x)\| = \|x - P_{\sigma_1^c}(x)\| > \rho$. Hence,
\begin{equation}\label{inc_are_spread_1}
\|P_{\sigma} (x)\|^2 \ge \|P_{\sigma_1}(x)\|^2 - \|P_{\sigma_2^c}(x)\|^2 \ge \rho^2 - n\cdot \|P_{\sigma_2^c}(x)\|^2_{\infty} \ge \rho^2/2. 
\end{equation}
On the other hand, in view of the inclusion $\sigma(x) \subset \sigma_1$, we get
\begin{equation}\label{inc_are_spread_2}
\|P_{\sigma}(x)\|^2 \le \|P_{\sigma}(x)\|_{\infty}^2 \cdot |\sigma| \le \frac{1}{\theta n}\cdot |\sigma|. 
\end{equation}
Together  \eqref{inc_are_spread_1} and \eqref{inc_are_spread_2} imply that $|\sigma| \ge \frac{1}{2}\rho^2\theta n$.
 \end{proof}

For incompressible vectors we will need the following basic estimate from \cite{RV square}.
\begin{prop}[{\cite[Lemma~3.5]{RV square}}]\label{invertibility via distance}
Let $M$ be a random $n\times n$
matrix with column vectors $X^1$, $X^2,\dots, X^n$,
and let $H_j$ ($j=1,2,\dots,n$) be the span of all column vectors except the $j$-th. Then for every $\varepsilon > 0$ we have 
$$
\P\bigl\{\inf\limits_{y \in \Incomp(\theta,\rho)} \|My\|<\varepsilon \rho n^{-1/2}\bigr\}
\le \frac{1}{\theta n} \sum_{j=1}^n\P \{\dist(X^j, H_j) < \varepsilon\}.
$$ 
\end{prop}
In view of independence and equi-measurability of the columns of $A$ in our model, the above proposition
yields for any $\varepsilon>0$:
$$\P\bigl\{\inf\limits_{y \in \Incomp(\theta,\rho)} \|Ay\|<\varepsilon \rho n^{-1/2}\bigr\}
\le \frac{1}{\theta}\P\Bigl\{\Bigl|\sum\limits_{i=1}^n X_i^*a_{in}\Bigr| < \varepsilon\Bigr\},$$
where $X^*=(X_1^*,X_2^*,\dots,X_n^*)$ denotes a random normal unit vector to the span of the
first $n-1$ columns of $A$.
Obtaining small ball probability estimates for $\Bigl|\sum\limits_{i=1}^n X_i^*a_{in}\Bigr|$ was
a crucial ingredient of \cite{RV square}.

Given a real-valued random variable $\xi$, define its {\it Levy concentration function} is
$$\Levy(\xi,z):=\sup\limits_{\lambda\in\R}\P\{|\xi-\lambda|\le z\},\;\;z\ge 0.$$
First, let us look at some well known estimates of
$\Levy(\xi,v)$ and then state a stronger bound from \cite{RV square}.

\begin{theor}[{Rogozin, \cite{Rogozin}}]\label{t: Rogozin}
Let $n\in\N$, let $\xi_1,\xi_2,\dots,\xi_n$ be jointly independent random variables
and let $t_1,t_2,\dots,t_n$ be some positive real numbers.
Then for any $t\geq \max\limits_{j}t_j$ we have
$$\Levy\Big(\sum_{j=1}^n\xi_j,t\Big)\leq C_{\ref{t: Rogozin}}\,t\,\Big(\sum_{j=1}^n(1-\Levy(\xi_j)){t_j}^2\Big)^{-1/2},$$
where $C_{\ref{t: Rogozin}}>0$ is a universal constant.
\end{theor}

Obviously, if $\xi$ is essentially non-constant, there are $v>0$ and $u\in(0,1)$ such that
$\Levy(\xi,v)\le u$. The following lemma is an elementary consequence of Theorem~\ref{t: Rogozin}
(see \cite[Lemma~3.6]{LPRT} and \cite[Lemma~2.6]{RV square} for similar statements proved
under additional moment assumptions on the variable).

\begin{lemma}\label{basic anticonc}
Let $\xi$ be a random variable with $\Levy(\xi,\widetilde v)\le\widetilde u$
for some $\widetilde v>0$ and $\widetilde u\in(0,1)$. Then there are $v'>0$ and $u'\in(0,1)$ depending only on $\widetilde u,\widetilde v$
with the following property: Let $\xi_1,\xi_2,\dots, \xi_n$ be independent copies of $\xi$.
Then for any vector $y\in S^{n-1}$ we have
$$\Levy\Bigl(\sum\limits_{j=1}^n y_j \xi_j, v'\Bigr)\le u'.$$
\end{lemma}
\begin{proof}
By Theorem~\ref{t: Rogozin},
for any $y\in S^{n-1}$
and any $h\ge\max\limits_{j}|y_j|\widetilde v$, we have
$$\Levy\Bigl(\sum\limits_{j=1}^n y_j \xi_j,h\Bigr)
\le \frac{C_{\ref{t: Rogozin}} h}{\widetilde v\sqrt{1-\widetilde u}}.$$
Define $v':=\frac{\widetilde v\sqrt{1-\widetilde u}}{2C_{\ref{t: Rogozin}}}$ and consider two cases.

\noindent 1) For every $j = 1, \ldots, n$ we have $|y_j|\le \frac{\sqrt{1-\widetilde u}}{2C_{\ref{t: Rogozin}}}$.
Then $v'\ge\max\limits_{j}|y_j|\widetilde v$, and
we obtain from the above relation
$$\Levy\Bigl(\sum\limits_{j=1}^n y_j \xi_j,v'\Bigr)
\le \frac{1}{2}.$$

\noindent 2) There is $j_0$ such that $|y_{j_0}|>\frac{\sqrt{1-\widetilde u}}{2C_{\ref{t: Rogozin}}}$. Then we get
$$\Levy\Bigl(\sum\limits_{j=1}^n y_j \xi_j,v'\Bigr)
\le \Levy(y_{j_0} \xi_{j_0},v')\le \widetilde u.$$

\noindent Thus, we can take $u':=\max(1/2,\widetilde u)$.
\end{proof}

\begin{lemma}[{``Tensorization lemma'', \cite[Lemma~2.2]{RV square}}]\label{l: RV tensorization}
Let $\alpha_1,\alpha_2,\dots,\alpha_n$ be i.i.d.\ random variables, and let $\varepsilon_0>0$.
\begin{itemize}

\item Assume that
$$\Levy(\alpha_1,\varepsilon)\leq L\varepsilon\quad\quad\mbox{for some $L>0$ and for all }\varepsilon\geq \varepsilon_0.$$
Then
$$\P\Big\{\sum_{j=1}^n {\alpha_j}^2\leq \varepsilon^2 n\Big\}\leq (CL\varepsilon)^n\quad\quad\mbox{for all }\varepsilon\geq \varepsilon_0,$$
where $C>0$ is a universal constant.

\item Assume that $\Levy(\alpha_1,v')\leq u'$ for some $v'>0$ and $u'\in(0,1)$.
Then there are $v>0$ and $u\in(0,1)$ depending only on $u',v'$ such that
$$\P\Big\{\sum_{j=1}^n {\alpha_j}^2\leq v n\Big\}\leq u^n.$$
\end{itemize}
\end{lemma}

As a consequence of Lemmas~\ref{basic anticonc} and~\ref{l: RV tensorization},
we get

\begin{lemma}\label{tensorization lemma}
Let $\alpha$ be a random variable with $\Levy(\alpha,\widetilde v)\le\widetilde u$
for some $\widetilde v>0$ and $\widetilde u\in(0,1)$. Then there are $v>0$ and $u\in(0,1)$ depending only on $\widetilde u,\widetilde v$
with the following property:
Let $A$ be an $n\times n$ random matrix with i.i.d.\ entries equidistributed with $\alpha$.
Then for any $y\in S^{n-1}$ we have
$$\P\big\{\|Ay\|\le v\sqrt{n}\big\}\le u^n.$$
\end{lemma}

\begin{rem}
Lemma~\ref{tensorization lemma} can be compared with \cite[Proposition~3.4]{LPRT} and \cite[Corollary~2.7]{RV square};
however, those statements were proved with additional assumptions on the entries of $A$.
\end{rem}

\bigskip

To get a stronger estimate than the one obtained in Lemma~\ref{basic anticonc},
the following notion was developed in \cite{RV square} and \cite{RV rectangular}
(see also preceding work \cite{TV inverse LO} by Tao and Vu).
\begin{definition}[Essential least common denominator]
For parameters $r \in (0,1)$ and $h > 0$ and any non-zero vector $x\in\R^n$, define
$$\LCD_{h,r}(x):= \inf\bigl\{t > 0: \dist(t x, \Z^n) < \min(r \|t x\|, h)\bigr\}.$$
\end{definition}
We note that later we shall choose $r$ sufficiently small and $h$ to be a small multiple of $\sqrt{n}$.
Thus, most of the coordinates of $\LCD_{h,r}(x)\cdot x$ are within a small distance to integers.
For a detailed discussion of the above notion, we refer to \cite{Rudelson}.

The next statement is proved in \cite{RV rectangular}.
\begin{theor}[{\cite[Theorem~3.4]{RV rectangular}}]\label{small ball probability}
Let $\xi_1,\xi_2,\dots,\xi_n$ be independent copies of a centered random variable such that
$\Levy(\xi_i,v)\le u$ for some $v>0$ and $u\in(0,1)$. Further,
let $x = (x_1, x_2,\dots, x_n) \in S^{n-1}$ be a fixed vector.
Then for every $h>0$, $r\in(0,1)$ and for every
$$\varepsilon \ge \frac{1}{\LCD_{h,r}(x)},$$
we have
 $$
\Levy\Bigl(\sum\limits_{i=1}^n x_i \xi_{i},\varepsilon v\Bigr)
\le\frac{C_{\text{\tiny\ref{small ball probability}}}\varepsilon}{r\sqrt{1-u}} +C_{\text{\tiny\ref{small ball probability}}} \exp\bigl(-2(1-u) h^2\bigr),
 $$
where $C_{\text{\tiny\ref{small ball probability}}}$ is a universal constant.
\end{theor}
 
Thus, in order to get a satisfactory small ball probability estimate for the infimum over incompressible vectors,
it is sufficient to show that the random normal $X^*$ has exponentially large $\LCD$ with probability close to one.
This will be done in the second part of Section~\ref{ssv section}. As for the set $\Comp$, our treatment
of the random normal will be based on results of Section~\ref{parallel section}.

\section{The smallest singular value --- proof of Theorem~B}\label{ssv section}

In this section we give a proof of Theorem~B stated in the introduction.
Let us start with a version of Theorem~A more convenient for us:
\begin{theoA'}
Let $\delta\in(0,1/4]$, $n\ge \frac{1}{4\delta}$, $\varepsilon\in(0,1/2]$, $S\subset S^{n-1}$, and let
$\Net\subset S$ be a Euclidean $\varepsilon$-net on $S$. Then there exists a (deterministic) subset $\widetilde\Net\subset S$
with $|\widetilde\Net|~\le~\exp\bigl(13\delta n\ln\frac{2e}{\delta}\bigr)|\Net|$ such that for any
$n\times n$ random matrix $A$ satisfying \eqref{probab model},
with probability at least $1-4\exp(-\delta n/8)$ the set $\widetilde\Net$
is a $(\frac{\varepsilon C_{\star}}{\delta}\sqrt{n})$--net on $S$
with respect to the pseudometric $d(x,y):=\|A(x-y)\|$ ($x,y\in S^{n-1}$), where $C_{\star}>0$ is a universal constant.
\end{theoA'}
\begin{proof}
Fix parameters $n$ and $\delta$, and
let $\mathcal C$ be the collection of parallelepipeds from Theorem\,A covering $B_2^n$.
Define a set
$\widetilde{\mathcal C}:=\big\{\varepsilon P+y:\,P\in\mathcal C,\;y\in\Net,\;S\cap(\varepsilon P+y)\neq \emptyset\big\}$
and
for every $\widetilde P\in\mathcal C$ let $y_{\widetilde P}$ be a point in the intersection $S\cap\widetilde P$.
Finally, set $\widetilde\Net:=\{y_{\widetilde P}:\,\widetilde P\in\widetilde{\mathcal C}\}$.
Informally speaking, $\widetilde{\mathcal C}$ is a ``product'' of the rescaled collection $\varepsilon\cdot\mathcal C$
and the net $\Net$. For each parallelepiped in $\widetilde {\mathcal C}$ having a non-empty intersection with $S$,
we take one (arbitrary) point from this intersection to construct the refined net $\widetilde\Net$.
What remains is to check that with high probability $\widetilde\Net$ is indeed a $(\frac{\varepsilon C}{\delta}\sqrt{n})$--net on $S$
with respect to the pseudometric $d(x,y):=\|A(x-y)\|$.

Observe that
$$|\widetilde\Net|=|\widetilde{\mathcal C}|\leq |\mathcal C|\cdot |\Net|\leq \exp\bigl(13\delta n\ln\frac{2e}{\delta}\bigr)|\Net|.$$
Next, let $A$ be an $n\times n$ random matrix satisfying \eqref{probab model}, and define event $\Event$ as
$$\Event:=\Big\{\forall \;x\in B_2^n\;\;\;\;\exists P\in\mathcal C\mbox{ such that }
x\in P\mbox{ and }A(P)\subset Ax+\frac{C\sqrt{n}}{\delta}B_2^n\Big\}.$$
By Theorem\,A, we have $\P(\Event)\geq 1-4\exp(-\delta n/8)$.

Fix any point $x\in S$. By the definition of $\Net$, there is a vector $y\in\Net$ such that $\varepsilon^{-1}(x-y)\in B_2^n$.
Hence, for any point $\omega\in\Event$ on the probability space, there is a parallelepiped $P=P(\omega)\in\mathcal C$
such that $\varepsilon^{-1}(x-y)\in P$ and
$$A_\omega(P)\subset A_\omega\big(\varepsilon^{-1}(x-y)\big)+\frac{C\sqrt{n}}{\delta}B_2^n.$$
Note that $S\cap(\varepsilon P+y)\supset \{x\}\neq\emptyset$, whence $\widetilde P:=\varepsilon P+y\in\widetilde{\mathcal C}$,
and, from the above relation,
$$A_\omega(\widetilde P)\subset A_\omega x+\frac{\varepsilon C\sqrt{n}}{\delta}B_2^n,$$
whence
$$A_\omega y_{\widetilde P}-A_\omega x\subset\frac{\varepsilon C\sqrt{n}}{\delta}B_2^n,$$
where $y_{\widetilde P}\in\widetilde \Net$.
We have shown that
$$\Event\subset\Big\{\forall x\in S\;\exists y=y(x)\in\widetilde \Net\mbox{ such that }\|A(x-y)\|\leq \frac{\varepsilon C\sqrt{n}}{\delta}\Big\},$$
and the result follows.
\end{proof}
\begin{rem}
Let us note that a weaker version of Theorem~$A^\star$, with condition $\widetilde\Net\subset S$ dropped,
can be proved by applying Corollary~A instead of Theorem~A.
\end{rem}

At this point, a significant part of our argument follows the same scheme as in \cite{RV square}.
In the first part of this section, we are dealing with compressible vectors.

\begin{prop}[Compressible vectors]\label{compressible prop}
Let $\alpha$ be a centered random variable with unit variance such that
$\Levy(\alpha,\widetilde v)\le\widetilde u$ for some $\widetilde v>0$ and $\widetilde u\in(0,1)$.
Then there are numbers $\theta_{\text{\tiny\ref{compressible prop}}},v_{\text{\tiny\ref{compressible prop}}}>0$
and $u_{\text{\tiny\ref{compressible prop}}}\in(0,1)$ depending only on $\widetilde v,\widetilde u$
with the following property:
Let $n\in\N$ and let $A$ be an $n\times n$ random matrix with i.i.d.\ entries equidistributed with $\alpha$.
Then for $\Comp=\Comp_n(\theta_{\text{\tiny\ref{compressible prop}}},\theta_{\text{\tiny\ref{compressible prop}}})$ we have
$$\P\bigl\{\inf\limits_{y\in\Comp}\|Ay\|
<v_{\text{\tiny\ref{compressible prop}}}\sqrt{n}\bigr\}
\le 5\,{u_{\text{\tiny\ref{compressible prop}}}}^n.$$
\end{prop}
\begin{proof}
Without loss of generality, we can assume that $n$ is large.
First, note that by Lemma~\ref{tensorization lemma} we have a strong probability
estimate for any fixed unit vector: there are $v>0$ and $u\in(0,1)$
depending on $\widetilde v,\widetilde u$ such that for any $y\in S^{n-1}$ we get
\begin{equation}\label{eq: aux 999}
\P\{\|Ay\|<v\sqrt{n}\}\le u^n.
\end{equation}
In order to obtain a uniform estimate over a set $S = \Comp_n(\theta,\theta)$ for some small parameter $\theta$,
we will take a net $\Net \subset S$ constructed in Lemma~\ref{estimating comp}
and refine it with the help of Theorem~$A^\star$ to get a net $\widetilde\Net$ with respect to pseudometric $\|A(x-y)\|$.
We will apply Theorem~$A^\star$ with parameter $\delta$ defined as the largest number in $(0,1/4]$
so that $ \exp\bigl(13\delta n\ln\frac{2e}{\delta}\bigr) \le u^{-n/3}$.
Let us describe the procedure in more detail.

First, define parameter $\theta\in(0,1/6]$ as the largest number satisfying the inequalities
$$\left(\frac{5 e}{\theta^2}\right)^{\theta n} \le u^{-n/3}\;\;\mbox{and}\;\;\frac{3 \theta C_{\star}}{\delta} \le \frac{v}{2}.$$
Let $S$ be as above. By Lemma~\ref{estimating comp}, there is a $3 \theta$-net $\Net\subset S$
on $S$ (with respect to the usual Euclidean metric) of cardinality $|\Net|\leq (\frac{5e}{\theta^2})^{\theta n}$.
Now, by Theorem~$A^\star$,
there is a {\it deterministic} subset $\widetilde\Net \subset S$ having the following properties:
\begin{itemize}

\item
$|\widetilde\Net| \le  \exp\bigl(13\delta n\ln\frac{2e}{\delta}\bigr)\cdot|\Net|
\le u^{-n/3}\cdot \big(\frac{5 e}{\theta^2}\big)^{\theta n} \le u^{-2n/3};$

\item With probability at least $1 - 4\exp(-\delta n/8)$ for every $y \in S$ there exists $x(y) \in \widetilde\Net$ such that
$$ \|A(x - y)\| \le \frac{3\theta\cdot C_{\star}}{\delta} \sqrt{n} \le \frac{v}{2}\sqrt{n}.$$
\end{itemize}

Applying the union bound over $\widetilde\Net$ to relation \eqref{eq: aux 999}, we get
$$\P\{\|Ay'\| < v\sqrt{n}\mbox{ for some }y'\in\widetilde\Net\}\le |\widetilde\Net|u^n\leq u^{n/3}.$$
On the other hand,
the second property of $\widetilde\Net$ implies that
$$\P\Big\{\inf\limits_{y\in S}\|Ay\| < \inf\limits_{y\in \widetilde\Net}\|Ay\|-\frac{v\sqrt{n}}{2}\Big\}
\leq 4\exp(-\delta n/8).$$
Combining the two estimates, we get
$$\P\{\|Ay\|< v\sqrt{n}/2\mbox{ for some }y\in S\}\le u^{n/3}+4\exp(-\delta n/8),$$
and the result follows with $u_{\text{\tiny\ref{compressible prop}}} := \max\{u^{1/3}, \exp(-\delta/8)\}$.
\end{proof}
\begin{rem}\label{n to n minus 1 rem}
It is not difficult to see that Proposition~\ref{compressible prop} can be stated and proved in the same way for $A$ which is not square, but instead is an $n-1\times n$ matrix with i.i.d.\ entries equidistributed with $\alpha$. Indeed, for $n$ large enough we can assume that $\gamma \cdot n < (n-1) < n$ for $\gamma$ as close to one as we want (the values of $\theta_{\text{\tiny\ref{compressible prop}}}$, $u_{\text{\tiny\ref{compressible prop}}}$
and $v_\text{\tiny{\ref{compressible prop}}}$ may differ in that case). This will
be important for us later.
\end{rem}
\begin{rem}
Proposition~\ref{compressible prop} could be proved by a completely different argument
based on \cite[Proposition~13]{T limit} and not using results of Section~\ref{parallel section} at all.
However, we prefer to have a ``uniform'' treatment of both compressible and incompressible vectors.
\end{rem}

\bigskip

Let us turn to estimating the infimum over
incompressible vectors. As we already discussed in Section~\ref{sec prelim},
it suffices to show that the random unit normal vector to the span of the first $n-1$ columns of
$A$ has exponentially large $\LCD$ with probability very close to one.
This property is verified in Theorem~\ref{random normal} below.
We start with some auxiliary statements.
First, note that Theorem~\ref{small ball probability}
together with Lemma~\ref{l: RV tensorization}
imply that anti-concentration probability for a single vector can be estimated in terms of
the LCD of the vector. Namely, the bigger $\LCD(x)$ is, the less is the probability that the image $Ax$ concentrates in a small ball:
\begin{lemma}[{Small ball probability for a single vector; see \cite[Lemma~5.5]{RV square}}]\label{small ball single vector}
Let $h>0$, $r\in(0,1)$ and let $\alpha$ be a random variable satisfying
$\Levy(\alpha,\widetilde v)\le\widetilde u$ for some $\widetilde v>0$ and $\widetilde u\in(0,1)$. Then there is
$L_{\text{\tiny\ref{small ball single vector}}}\geq 1$ depending only on $\widetilde v,\widetilde u$
with the following property: Let $A'$ be an $n-1\times n$ random matrix with i.i.d.\ elements equidistributed with $\alpha$.
Then for any vector $x\in S^{n-1}$ and any 
$$\varepsilon\ge \widetilde v \cdot \max\Bigl(\frac{1}{\LCD_{h,r}(x)},
\exp(-2(1-\widetilde u) h^2)\Bigr)
$$ 
we have
$$\P\{\|A'x\|< \varepsilon\sqrt{n}\}\le (L_{\text{\tiny\ref{small ball single vector}}}\varepsilon/r)^{n-1}.$$
\end{lemma}
\begin{proof}
Fix any vector $x\in S^{n-1}$ and denote $Y=(Y_1,Y_2,\dots,Y_{n-1}):=A'x$.
Note that, in view of Theorem~\ref{small ball probability}, we have
$$
\Levy\bigl(Y_i,\varepsilon\bigr)
\leq \frac{C_{\text{\tiny\ref{small ball probability}}}\varepsilon}{r\sqrt{1-\widetilde u}}
+C_{\text{\tiny\ref{small ball probability}}} \exp\bigl(-2(1-\widetilde u) h^2\bigr)
\leq\frac{C_{\text{\tiny\ref{small ball probability}}}(1+\widetilde v^{-1})\varepsilon}{r\sqrt{1-\widetilde u}},\quad\quad i\leq n,
$$
for any $\varepsilon$ satisfying conditions of the lemma.
Hence, by Lemma~\ref{l: RV tensorization},
$$\P\Big\{\sum_{i=1}^{n-1} {Y_i}^2\leq \varepsilon^2 (n-1)\Big\}
\leq \Big(\frac{C'(1+\widetilde v^{-1})\varepsilon}{r\sqrt{1-\widetilde u}}\Big)^{n-1}.$$
\end{proof}

The above statement is useful for incompressible vectors:
the following Lemma~\ref{incompressible lcd}
shows that incompressible vectors have $\LCD$ at least of order $\sqrt{n}$.
The lemma is taken from papers \cite{RV square, RV rectangular}, and its proof is included for completeness.
\begin{lemma}[see {\cite[Lemma~3.6]{RV rectangular}}]\label{incompressible lcd}
For every $\theta,\rho\in(0,1)$ there are
$q_{\text{\tiny\ref{incompressible lcd}}}=q_{\text{\tiny\ref{incompressible lcd}}}(\theta,\rho)>0$ and
$r_{\text{\tiny\ref{incompressible lcd}}}=r_{\text{\tiny\ref{incompressible lcd}}}(\theta,\rho)>0$ such that for
every $h>0$ any vector $x\in\Incomp_n(\theta,\rho)$ satisfies
$$\LCD_{h,r_{\text{\tiny\ref{incompressible lcd}}}}(x) \ge q_{\text{\tiny\ref{incompressible lcd}}} \sqrt{n}.$$
\end{lemma}
\begin{proof}
Set $a:=\frac{1}{2}\rho^2\theta$ and $b:=\rho/\sqrt{2}$.
We choose
$r=r_{\text{\tiny\ref{incompressible lcd}}}:=
b \sqrt{\frac{a}{2}}=\frac{1}{2}\rho^2\sqrt{\theta}$
and $q=q_{\text{\tiny\ref{incompressible lcd}}} := \big(1/\sqrt{\theta}
+ \frac{2 r}{a}\big)^{-1}=\sqrt{\theta}/3$.  
   
Let $x\in\Incomp_n(\theta,\rho)$, $h>0$ and assume that
$\LCD_{h,r}(x) < q \sqrt{n}$. Then, by definition of least common denominator,
there exist $p \in \Z^n$ and $\lambda \in (0, q \sqrt{n})$ such that
\begin{equation}\label{lcd}
\|\lambda x - p\| < r \lambda < r q \sqrt{n}=\frac{1}{6}\rho^2\theta\sqrt{n}=\frac{1}{3}a\sqrt{n}.
\end{equation}
It is easy to check that for a vector with such norm the set
$$\widetilde\sigma(x)
:= \big\{i \leq n: | \lambda x_i - p_i| < 2/3 \big\}$$
has a cardinality at least $(1 - \frac{a^2}{4}) n$.
Further, by Lemma \ref{incompressible are spread}, the set of ``spread'' coordinates $\sigma(x)$ has cardinality at least $a n$.
Hence, the set $I(x) := \sigma(x) \cap \widetilde\sigma(x)$ is non-empty, and $|I(x)| > \frac{a}{2}n$.
For any $i\in I(x)$ we have
$$ |p_i| < \lambda |x_i| + \frac{2}{3} < \frac{q}{\sqrt{\theta}}
+ \frac{2 r q}{a} = 1
$$
(in the last step we used our definition of $q$). Since $p \in \Z^n$, we get that $p_i = 0$ for all $i\in I(x)$. 

Finally, due to the definition of $I(x)$ and our choice of $r$, denoting by $P_J$ the coordinate projection on a span $\{i\in J:e_i\}$,
we obtain
$$
\|\lambda x - p\|^2 \ge\| \lambda P_{I}(x)\|^2
> \lambda^2 |I(x)| \frac{\rho^2}{2n}=\lambda^2\frac{\rho^2 a}{4}=(r \lambda)^2,$$
which contradicts \eqref{lcd} and, hence, the assumption that $\LCD_{h,r}(x) < q \sqrt{n}$.
\end{proof}

Let $n\in\N$, $h>0$, $\theta,\rho\in(0,1)$, and
let $q_{\text{\tiny\ref{incompressible lcd}}}$ and $r_{\text{\tiny\ref{incompressible lcd}}}$
be as in the above statement.
Following \cite{RV square}, we consider the ``level sets'' $S_k$ of $\Incomp_n(\theta,\rho)$ defined as
$$S_k=S_k(\theta,\rho,h):= \bigl\{x \in \Incomp_n(\theta,\rho)\,:\,k \le \LCD_{h,r_{\text{\tiny\ref{incompressible lcd}}}} (x) < 2k\bigr\},
\;\;k\ge 0.$$
In the proof of the theorem below we will partition $\Incomp_n(\theta,\rho)$ into subsets of vectors having $\LCD$'s of the same order:
\begin{equation}\label{eq: partition}
\Incomp_n(\theta,\rho)= \bigsqcup_{k = 2^i,  i \geq i_0} S_k,
\end{equation}
where, using Lemma~\ref{incompressible lcd},
we introduce the lower bound $i_0: = \log_2(q_{\text{\tiny\ref{incompressible lcd}}} \sqrt{n}/2)$
(we have $S_k=\emptyset$ for all $k < q_{\text{\tiny\ref{incompressible lcd}}} \sqrt{n}/2$).
Following \cite{RV square}, we are going to combine estimates for individual sets $S_k$.  

A principal observation made in \cite{RV rectangular} and \cite{RV square} is that the sets $S_k$ admit Euclidean $\varepsilon$-nets of relatively small cardinality.
We give both the formal statement and its proof from \cite{RV rectangular} below for the sake of completeness:
\begin{lemma}[{\cite[Lemma~4.8]{RV rectangular}}]\label{cardinality lemma}
For any $\theta,\rho\in(0,1)$ there is $L=L(\theta,\rho)>0$ such that for every $h\ge 1$ and
$k> 0$ the set $S_k$ admits a Euclidean $(4h/k)$-net of
cardinality at most $\bigl(k L/\sqrt{n}\bigr)^n$.
\end{lemma}
\begin{proof}
In view of Lemma~\ref{incompressible lcd}, we can assume that $k\geq q_{\text{\tiny\ref{incompressible lcd}}} \sqrt{n}/2$.
Further, without loss of generality $\frac{4h}{k} < 2$; otherwise a one-point net works.

Fix for a moment a point $x \in S_k$. Then, by definition of the ``level sets'',
$k \le \LCD_{h,r_{\text{\tiny\ref{incompressible lcd}}}} (x) < 2k$. By definition of LCD, there exists $p=p(x) \in \Z^n$
such that
$$\| \LCD_{h,r_{\text{\tiny\ref{incompressible lcd}}}} (x) \cdot x - p \| \leq h.$$
Hence,
$$\Big\| x - \frac{p}{\LCD_{h,r_{\text{\tiny\ref{incompressible lcd}}}} (x) } \Big\| \leq \frac{h}{k} < \frac{1}{2}.$$
It is a simple planimetric observation that if we normalize the vector
$p/\LCD_{h,r_{\text{\tiny\ref{incompressible lcd}}}} (x)$, the distance to the unit vector $x$ cannot increase more than twice:
$$ \Big\| x - \frac{p}{\|p\|}\Big\| \leq \frac{2h}{k}.$$
Thus, the set
$$\Net_{int}:=\Big\{\frac{p}{\|p\|}:\,p=p(x)\mbox{ for some }x\in S_k\Big\}$$
is a $2h/k$-net for $S_k$.
How many different $p \in \Z^n$ we have to consider?
Note that for any $x\in S_k$, the norm of $p(x)$ cannot be too large:
since $\|x\| = 1$, $\LCD_{h,r_{\text{\tiny\ref{incompressible lcd}}}} (x) < 2k$ and $4h/k < 2$, we get
$$ \|p(x)\| \leq \LCD_{h,r_{\text{\tiny\ref{incompressible lcd}}}} (x) + h < 3k.$$
Hence, all vectors $p\in\Z^n$ in the definition of $\Net_{int}$
belong to the Euclidean ball of radius $3k$ centered at the origin.
Standard volumetric argument shows that there are at most $(1 + Ck/\sqrt{n})^{n}$ integer points in this ball
for a sufficiently large constant $C>0$.
Recall that $k \ge q_{\text{\tiny\ref{incompressible lcd}}} \sqrt{n}/2$, whence
$$|\Net_{int}| \le \Big(1 + \frac{Ck}{\sqrt{n}}\Big)^{n} \le \Bigl(\frac{kL}{\sqrt{n}}\Bigr)^n$$
for an appropriate number $L=L(\theta,\rho)>0$.
The net $\Net_{int}$ does not have to be contained in $S_k$. But, by a standard argument,
we can ``replace'' $\Net_{int}$ with a $4h/k$-net of the same cardinality, and with elements from the set $S_k$.
\end{proof}

Together with Theorem~$A^\star$, the above lemma gives
\begin{lemma}\label{net on level sets}
For any $\theta,\rho\in(0,1)$ there is $L_{\text{\tiny\ref{net on level sets}}}=L_{\text{\tiny\ref{net on level sets}}}(\theta,\rho)\geq 1$
such that for every $h\ge 1$ and $k> 0$ there is a finite subset $\Net\subset S_k$
of cardinality at most $\bigl(k L_{\text{\tiny\ref{net on level sets}}}/\sqrt{n}\bigr)^n$ with the following property.
The event
$$\bigl\{\mbox{For every $y\in S_k$ there is $y'=y'(y)\in\Net$ such that $\|A(y-y')\|
\le h L_{\text{\tiny\ref{net on level sets}}}\sqrt{n}/k$}\bigr\}$$
has probability at least $1-4\exp(-n/32)$.
\end{lemma}

Now, we can prove
\begin{theor}\label{random normal}
Let $\alpha$ be a centered random variable of unit variance such that
$\Levy(\alpha,\widetilde v)\le\widetilde u$ for some $\widetilde v>0$ and $\widetilde u\in(0,1)$.
Then there exist $q,s,w,r> 0$ depending only on
$\widetilde v,\widetilde u$ with the following property: let
$X^1, X^2\dots, X^{n-1}$ be random $n$-dimensional vectors whose coordinates are jointly independent copies of $\alpha$.
Consider any random unit vector $X^*$ orthogonal to $\{X^1,X^2,\dots,X^{n-1}\}$.
Then
$$\P\big\{\LCD_{s\sqrt{n},r}(X^*)<\exp(q n)\big\} \le 2\exp(-w n).$$
\end{theor}
\begin{proof}
Without loss of generality, we can assume that $n$ is a large number and that $\widetilde v\leq 1$.
Denote by $A'$ the $n-1\times n$ matrix with rows $X^1,X^2,\dots,X^{n-1}$. Then, by the definition of $X^*$, we have $A'X^*=0$
almost surely.
Let $\theta_{\text{\tiny\ref{compressible prop}}}$ and $u_{\text{\tiny\ref{compressible prop}}}$
be defined as in Remark~\ref{n to n minus 1 rem} (with $A'$ replacing $A$).
Then, by Proposition~\ref{compressible prop} and Remark~\ref{n to n minus 1 rem}, we have
$$\P\bigl\{X^*\in\Comp_n(\theta_{\text{\tiny\ref{compressible prop}}},
\theta_{\text{\tiny\ref{compressible prop}}})\bigr\}\le 5{u_{\text{\tiny\ref{compressible prop}}}}^n \leq \exp(-w n)
$$
for $w>0$ such that, say,
$\exp(-2w) > u_{\text{\tiny\ref{compressible prop}}}$, and provided that $n$ is large. Thus, it is enough to prove that
$$\P\big\{\LCD_{s\sqrt{n},r}(X^*)<\exp(q n), X^* \in \Incomp_n(\theta_{\text{\tiny\ref{compressible prop}}},
\theta_{\text{\tiny\ref{compressible prop}}}) \big\} \le \exp(-w n)$$
for small enough $r,w,s,q$ depending only on $\widetilde v,\widetilde u$.
We start by defining $r:=r_{\text{\tiny\ref{incompressible lcd}}}
(\theta_{\text{\tiny\ref{compressible prop}}},\theta_{\text{\tiny\ref{compressible prop}}})$.
Note that, by Lemma~\ref{incompressible lcd}, we have
$$\Incomp_n(\theta_{\text{\tiny\ref{compressible prop}}},
\theta_{\text{\tiny\ref{compressible prop}}})\subset \big\{x\in S^{n-1}:\,\LCD_{s\sqrt{n},r}(x)
\geq q_{\text{\tiny\ref{incompressible lcd}}}\sqrt{n}\big\}$$
for {\it any} $s>0$, and, in particular for $s$ defined by
$s:=\frac{\widetilde v r}{4L^2_{\text{\tiny\ref{net on level sets}}}L_{\text{\tiny\ref{small ball single vector}}}}$,
where $L_{\text{\tiny\ref{net on level sets}}}=L_{\text{\tiny\ref{net on level sets}}}(\theta_{\text{\tiny\ref{compressible prop}}},
\theta_{\text{\tiny\ref{compressible prop}}})$ and $L_{\text{\tiny\ref{small ball single vector}}}$
are taken from Lemmas~\ref{net on level sets} and~\ref{small ball single vector}, respectively,
and $q_{\text{\tiny\ref{incompressible lcd}}}=q_{\text{\tiny\ref{incompressible lcd}}}
(\theta_{\text{\tiny\ref{compressible prop}}},\theta_{\text{\tiny\ref{compressible prop}}})$.
Let us emphasize that no vicious cycle is created here in regard to interdependence between $s$ and $r$.
Finally, we let $q:=2s^2(1-\widetilde u)$ ($w$ will be defined at the very end of the proof).

We will make use of representation \eqref{eq: partition} of the set
$\Incomp_n(\theta_{\text{\tiny\ref{compressible prop}}},\theta_{\text{\tiny\ref{compressible prop}}})$.
Denote
$$\K :=  \big\{2^i: \;i\in[\log_2(q_{\text{\tiny\ref{incompressible lcd}}}\sqrt{n})-1,qn/\ln 2]\cap \N\big\}.$$
Then, in view of Lemma~\ref{incompressible lcd}, we have
$$\big\{x\in \Incomp_n(\theta_{\text{\tiny\ref{compressible prop}}},\theta_{\text{\tiny\ref{compressible prop}}}):\,
\LCD_{s\sqrt{n},r}(x)<\exp(q n)\big\}\subset\bigsqcup_{k\in \K}S_k.$$
It is sufficient to prove that 
\begin{equation}\label{on_one_level}
\P\{X^* \in S_k \} \le 5\exp(-n/32) \quad \text{ for all } k \in \K.
\end{equation}
Indeed, since $|\K| < qn$, the union bound over $\K$ will conclude the theorem.

In turn, \eqref{on_one_level} will follow as long as we show that
$$\P\{A'x=0\mbox{ for some }x\in S_k \} \le 5\exp(-n/32) \quad \text{ for all } k \in \K.$$

Fix for a moment any $k\in\K$ and let 
$\Net_k$ be the subset of $S_k$ of cardinality at most $(kL_{\text{\tiny\ref{net on level sets}}}/\sqrt{n})^n$,
constructed in Lemma~\ref{net on level sets} (with $h:=s\sqrt{n}$).
Further, take $\eps := \frac{\widetilde v r \sqrt{n}}
{2k L_{\text{\tiny\ref{net on level sets}}}L_{\text{\tiny\ref{small ball single vector}}}}$.
Note that, in view of the definition of $q$ and $\K$, we have
$k\leq \exp(2s^2(1-\widetilde u)n)$. Hence, for
$n$ large enough,
$\eps$ satisfies the condition of Lemma~\ref{small ball single vector}:
$$\varepsilon\geq\widetilde v \cdot \max\Bigl(\frac{1}{k}, \exp\big(-2s^2(1-\widetilde u)n\big)\Bigr)
\geq\widetilde v \cdot \max\Bigl(\frac{1}{\LCD_{h,r}(x)}, \exp(-2(1-\widetilde u) h^2)\Bigr).$$
Hence, 
\begin{align*}
\P\big\{\|A'y\|\ge \varepsilon\sqrt{n}\mbox{ for all }y\in\Net_k\big\}
&\ge 1-|\Net_k|(L_{\text{\tiny\ref{small ball single vector}}}\varepsilon/r)^{n-1}\\
&\ge 1-\Bigl(\frac{k L_{\text{\tiny\ref{net on level sets}}}}{\sqrt{n}}\Bigr)^n
\Big(\frac{L_{\text{\tiny\ref{small ball single vector}}}\varepsilon}{r}\Big)^{n-1}\\
& \ge 1 - \frac{k L_{\text{\tiny\ref{net on level sets}}}}{\sqrt{n}} \cdot
\Bigl(\frac{\widetilde v}{2}\Bigr)^{n-1} \\
&\ge 1 - 2^{-n}\exp(2s^2(1-\widetilde u)n),
\end{align*}
where the last relation follows by the assumption $\widetilde v\leq 1$.
Finally, note that, since $s\leq 1/4$, the last quantity is bounded from below by $1-2^{-n/2}$.
Applying the definition of $\Net_k$ in Lemma~\ref{net on level sets} and noticing that
$hL_{\text{\tiny\ref{net on level sets}}}\sqrt{n}/k\leq \varepsilon\sqrt{n}/2$, we get
$$\P\big\{\|A'y\|\ge \varepsilon\sqrt{n}/2\mbox{ for all }y\in S_k\big\}
\geq 1-4\exp(-n/32)-2^{-n/2}\geq 1-5\exp(-n/32).$$
This proves \eqref{on_one_level} and implies the result.
\end{proof}

\begin{proof}[Proof of Theorem~B]
Without loss of generality, the dimension $n$ is large.
Let $A=(a_{ij})$ be an $n\times n$ random matrix with i.i.d.\ centered entries with unit variance
such that for some $\widetilde v>0$ and $\widetilde u\in(0,1)$ we have $\Levy(a_{ij},\widetilde v)\leq\widetilde u$.
We define $\theta:=\theta_{\text{\tiny\ref{compressible prop}}}(\widetilde v,\widetilde u)$
and $v:=v_{\text{\tiny\ref{compressible prop}}}(\widetilde v,\widetilde u)$,
where $\theta_{\text{\tiny\ref{compressible prop}}},v_{\text{\tiny\ref{compressible prop}}}$
are taken from Proposition~\ref{compressible prop}, and let
$q,s,w,r$ be as in Theorem~\ref{random normal} (with respect to $\widetilde v,\widetilde u$).
We will prove a small ball probability bound for $s_n(A)$.

It is sufficient to consider the parameter domain
$\varepsilon\in\big(\theta\widetilde v\exp(-qn),1\big]$.
We have
\begin{align*}
\P\bigl\{s_n(A)
<\varepsilon n^{-1/2}\bigr\}
&\leq
\P\bigl\{\inf\limits_{y\in\Comp_n(\theta,\theta)}\|Ay\|
<v\sqrt{n}\bigr\}
+\P\bigl\{\inf\limits_{y\in\Incomp_n(\theta,\theta)}\|Ay\|
<\varepsilon n^{-1/2}\bigr\}\\
&\leq 5{u_{\text{\tiny\ref{compressible prop}}}}^n
+\P\bigl\{\inf\limits_{y\in\Incomp_n(\theta,\theta)}\|Ay\|
<\varepsilon n^{-1/2}\bigr\},
\end{align*}
where we have applied Proposition~\ref{compressible prop}.
Further, by Proposition~\ref{invertibility via distance}, we have
$$\P\bigl\{\inf\limits_{y\in\Incomp_n(\theta,\theta)}\|Ay\|
<\varepsilon n^{-1/2}\bigr\}\leq \frac{1}{\theta}
\P\Bigl\{\Bigl|\sum\limits_{i=1}^n X_i^*a_{in}\Bigr| < \frac{\varepsilon}{\theta}\Bigr\},
$$
where $X^*$ denotes a random unit normal vector to the span of the first $n-1$ columns of $A$.
In view of Theorem~\ref{small ball probability}, this last relation implies
\begin{align*}
\P\bigl\{\inf\limits_{y\in\Incomp_n(\theta,\theta)}\|Ay\|
<\varepsilon n^{-1/2}\bigr\}&\leq
\theta^{-1}\P\bigl\{\LCD_{s\sqrt{n},r}(X^*)<\theta\widetilde v \varepsilon^{-1}\bigr\}\\
&+\frac{C_{\text{\tiny\ref{small ball probability}}}\varepsilon}{\theta\widetilde v r\sqrt{1-\widetilde u}}
+C_{\text{\tiny\ref{small ball probability}}} \exp\bigl(-2s^2(1-\widetilde u) n\bigr).
\end{align*}
Finally, noticing that $\theta\widetilde v \varepsilon^{-1}\leq \exp(qn)$
and applying Theorem~\ref{random normal},
we get
$$\P\bigl\{\inf\limits_{y\in\Incomp_n(\theta,\theta)}\|Ay\|
<\varepsilon n^{-1/2}\bigr\}\leq 2\theta^{-1}\exp(-wn)+\frac{C_{\text{\tiny\ref{small ball probability}}}\varepsilon}{\theta\widetilde v r\sqrt{1-\widetilde u}}
+C_{\text{\tiny\ref{small ball probability}}} \exp\bigl(-2s^2(1-\widetilde u) n\bigr).$$
Together with an estimate for the compressible vectors, this implies the result.
\end{proof}

\section{Acknowledgements}

We would like to thank N. Tomczak-Jaegermann and
R.~Vershynin for valuable suggestions that helped improve the presentation of the work.
The second named author is grateful to A.~Litvak for inspiring discussions.
Both authors are indebted to the referee for very useful remarks and suggestions.

\address


\begin{thebibliography}{99}

\bibitem{ALPT}{R. Adamczak, A. E. Litvak, A. Pajor, N. Tomczak-Jaegermann,
Quantitative estimates of the convergence of the empirical covariance matrix in Log-concave Ensembles,
Journal of AMS, 234 (2010), 535--561.}

\bibitem{ALPT2}{R. Adamczak, A. E. Litvak, A. Pajor, N. Tomczak-Jaegermann,
Sharp bounds on the rate of convergence of empirical covariance matrix,
C.R. Math. Acad. Sci. Paris, 349 (2011), 195--200.}

\bibitem{BY}
{
Z. D. Bai\ and\ Y. Q. Yin, Limit of the smallest eigenvalue of a large-dimensional sample covariance matrix, Ann. Probab. 21 (1993), no.~3, 1275--1294.
}

\bibitem{Edelman}
{
A. Edelman, Eigenvalues and condition numbers of random matrices, SIAM J. Matrix Anal. Appl. 9 (1988), no.~4, 543--560. 
}

\bibitem{GLPT}{O. Guedon, A. E. Litvak, A. Pajor, N. Tomczak-Jaegermann,
On the interval of fluctuation of the singular values of random matrices,
Journal of the European Mathematical Society, to appear.}

\bibitem{GNT}{F. G\"otze, A. Naumov, and A. Tikhomirov.
On minimal singular values of random matrices with correlated entries.
Random Matrices:
Theory Appl. 04, 1550006 (2015), DOI: 10.1142/S2010326315500069}

\bibitem{Hoeffding}{W.~Hoeffding, Probability inequalities for sums of bounded random variables,
Journal of the American Statistical Association 58 (1963), No.~301, 13--30.}

\bibitem{KM}{V. Koltchinskii; S. Mendelson,
Bounding the Smallest Singular Value of a Random Matrix Without Concentration,
Int Math Res Notices (2015), doi: 10.1093/imrn/rnv096.}

\bibitem{Latala}{R.~Latala, Some estimates of norms of random matrices, Proc. Amer. Math. Soc. 133 (2005), 1273--1282.}

\bibitem{LV}{C. Le, R. Vershynin, Concentration and regularization of random graphs, preprint. arXiv:1506.00669.}

\bibitem{LPRT}{Litvak,~A. E.; Pajor,~A.; Rudelson,~M.; Tomczak-Jaegermann,~N.
Smallest singular value of random matrices and geometry of random polytopes. Adv. Math. 195 (2005), no. 2, 491--523.}

\bibitem{MP}{S. Mendelson; G. Paouris,
On the singular values of random matrices, Journal of EMS, 16 (2014), no 4, 823--834.}

\bibitem{MS}{V.D.~Milman, G.~Schechtman, Asymptotic theory of finite-dimensional normed spaces, Lecture Notes
in Math., vol. 1200, Springer, Berlin, 1986.}

\bibitem{NR}{Hoi H. Nguyen and Sean O'Rourke. The Elliptic Law.
Int Math Res Notices, 2014, doi:10.1093/imrn/rnu174}

\bibitem{RW}{Rempala, Grzegorz; Wesolowski, Jacek.
Asymptotics for products of independent sums with an application to Wishart determinants.
Statist. Probab. Lett. 74 (2005), no. 2, 129--138.}

\bibitem{Rogozin}{Rogozin, B. A. On the increase of dispersion of sums of independent random variables.
(Russian) Teor. Verojatnost. i Primenen 6 (1961), 106--108.}

\bibitem{Rudelson}{M. Rudelson, Lecture notes on non-aymptotic random matrix theory,
Notes from the AMS Short Course on Random Matrices, 2013.}

\bibitem{RV congress}{Rudelson,~M.; Vershynin,~R. Non-asymptotic theory of random matrices: extreme singular values.
Proceedings of the International Congress of Mathematicians. Volume III, 1576--1602, Hindustan Book Agency, New Delhi, 2010.}

\bibitem{RV rectangular}{Rudelson,~M.; Vershynin,~R. Smallest singular value of a random rectangular matrix.
Comm. Pure Appl. Math. 62 (2009), no. 12, 1707--1739.}

\bibitem{RV square}{Rudelson,~M.; Vershynin,~R. The Littlewood-Offord problem and invertibility of random matrices.
Adv. Math. 218 (2008), no. 2, 600--633.}

\bibitem{SV}{N. Srivastava, R. Vershynin, Covariance estimation for distributions with $2+\varepsilon$ moments,
Annals of Probability 41 (2013), 3081--3111.}

\bibitem{TV inverse LO}{T. Tao, V. Vu,
Inverse Littlewood-Offord theorems and the condition number of
random discrete matrices, Ann. Math., 169 (2009), 595--632.}

\bibitem{Tao Vu Circular law}{T. Tao, V. Vu, Random matrices: the circular law, Commun. Contemp. Math. 10
(2008), no.~2, 261--307.}

\bibitem{TV distribution of smin}
{
T. Tao\ and\ V. Vu, Random matrices: the distribution of the smallest singular values,
Geom. Funct. Anal. 20 (2010), no.~1, 260--297.
}

\bibitem{TV smooth analysis}{Tao,~T.; Vu,~V. Smooth analysis of the condition number and the least singular value.
Math. Comp. 79 (2010), no. 272, 2333--2352.}

\bibitem{Thorisson}
{
H. Thorisson, {\it Coupling, stationarity, and regeneration}, Probability and its Applications (New York), Springer, New York, 2000.
}

\bibitem{T limit}{K. Tikhomirov, The limit of the smallest singular value of random matrices with i.i.d. entries,
Adv. Math. 284 (2015), 1--20.}

\bibitem{Tikhomirov}{K. E. Tikhomirov, The smallest singular value of random rectangular matrices with
no moment assumptions on entries, Israel J. Math. 212 (2016), no.~1, 289--314.}

\bibitem{Vershynin}{Vershynin,~R. Introduction to the non-asymptotic analysis of random matrices. In:
Compressed Sensing: Theory and Applications, Yonina Eldar and Gitta Kutyniok
(eds), 210--268, Cambridge University Press, 2012.}

\bibitem{Yas14}{P. Yaskov, Lower bounds on the smallest eigenvalue of a sample covariance matrix,
Electron. Commun. Probab. 19 (2014), no. 83, 1--10.}

\bibitem{YBK}{Y. Q. Yin, Z. D. Bai\ and\ P. R. Krishnaiah, On the limit of the largest eigenvalue
of the large-dimensional sample covariance matrix,
Probab. Theory Related Fields {\bf 78} (1988), no.~4, 509--521.}

\end{thebibliography}
\end{document}